\documentclass[11pt]{amsart}

\usepackage{hyperref}
\usepackage{graphicx,subfigure,epsfig}
\usepackage{subfigure}
\usepackage{epstopdf}
\usepackage{mathtools}
\DeclarePairedDelimiter\abs{\lvert}{\rvert}

\usepackage{amssymb,amsmath}
\usepackage{caption}
\usepackage{float}
\usepackage[usenames, dvipsnames]{color}
\usepackage{bm}
\usepackage{enumitem}
\newtheorem{theorem}{Theorem}[section]
\newtheorem{definition}{Definition}[section]
\newtheorem{lemma}{Lemma}[section]
\newtheorem{remark}{Remark}[section]

\newtheorem{example}{Example}[section]
\newtheorem{proposition}{Proposition}[section]

\newtheorem{problem}{Problem}[section]

\usepackage[utf8]{inputenc}
 
\begin{document}
\title{An unknotting invariant for welded knots}
\author{ K. Kaur} 
\address{Kirandeep Kaur\\ Department of Mathematics\\ Indian Institute of Technology Ropar\\
 Nangal Road, Rupnagar, Punjab 140001, INDIA } 
\email{kirandeep.kaur@iitrpr.ac.in}

\author{ A. Gill} 
\address{Amrendra Gill\\ Department of Mathematics\\ Indian Institute of Technology Ropar\\
 Nangal Road, Rupnagar, Punjab 140001, INDIA } 
\email{amrendra.gill@iitrpr.ac.in}

\author{ M. Prabhakar } 
\address{Madeti Prabhakar  \\Department of Mathematics\\ Indian Institute of Technology Ropar\\
 Nangal Road, Rupnagar, Punjab 140001, INDIA}  
\email{prabhakar@iitrpr.ac.in}

\author{ A. Vesnin} 
\address{Andrei Vesnin\\ Regional Mathematical Center, Tomsk State University, Tomsk, 634050, RUSSIA} 
\email{vesnin@math.nsc.ru}

\begin{abstract}
We study a local twist move on welded knots that is an analog of the virtualization move on virtual knots. Since this move is an unknotting operation we define an invariant, \emph{unknotting twist number}, for welded knots. We relate the unknotting twist number with warping degree and welded unknotting number, and establish a lower bound on the twist number using Alexander quandle coloring. We also study the Gordian distance between welded knots by twist move and define the corresponding Gordian complex.
\end{abstract}
\keywords{Welded knot, unknotting twist number, quandle coloring, Gordian distance}
\makeatletter{\renewcommand*{\@makefnmark}{}
\footnotetext{2010 {\it Mathematics Subject Classifications.} 57M25, 57M27.}\makeatother}

\maketitle

\section{Introduction}
The virtual knots and welded knots are two extensions of classical knots in the 3-sphere. 
Welded knots were first introduced by R.~Fenn, R.~Rim{\'a}nyi, and C.~Rourke in the study of braid-permutation group \cite{fenn1997braid}. Later L.~Kauffman \cite{kauffman1999virtual} established the virtual knot theory which gives an approach  to study knots embedded in certain 3-manifolds. A virtual knot diagram has two types of crossings: classical crossings and virtual crossings as shown in Fig.~\ref{fig1}. 

\begin{figure}[!ht]
\centering 
\unitlength=0.4mm
\begin{picture}(140,30)(-10,5)
\thicklines
\qbezier(-10,10)(-10,10)(10,30)
\qbezier(-10,30)(-10,30)(-2,22) 
\qbezier(10,10)(10,10)(2,18)
\qbezier(70,10)(70,10)(50,30)
\qbezier(70,30)(70,30)(50,10) 
\put(60,20){\circle{4}}
\qbezier(130,10)(130,10)(110,30)
\qbezier(130,30)(130,30)(110,10) 
\put(120,20){\circle*{4}}
\end{picture}
\caption{Classical, virtual and welded crossings.} \label{fig1}
\end{figure}
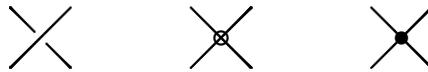

Two virtual knot diagrams represent the same virtual knot if one can be obtained from the other via a finite sequence of generalized Reidemeister moves  which are classical Reidemeister moves (RI, RII, and RIII) or virtual Reidemeister moves (VRI, VRII, VRIII, and SV) shown in Fig.~\ref{rm}, where for further reasons virtual crossings are drawn as welded crossings. 
\begin{figure}[!ht] {\centering
\subfigure[Classical Reidemeister moves.] {\includegraphics[scale=0.3]{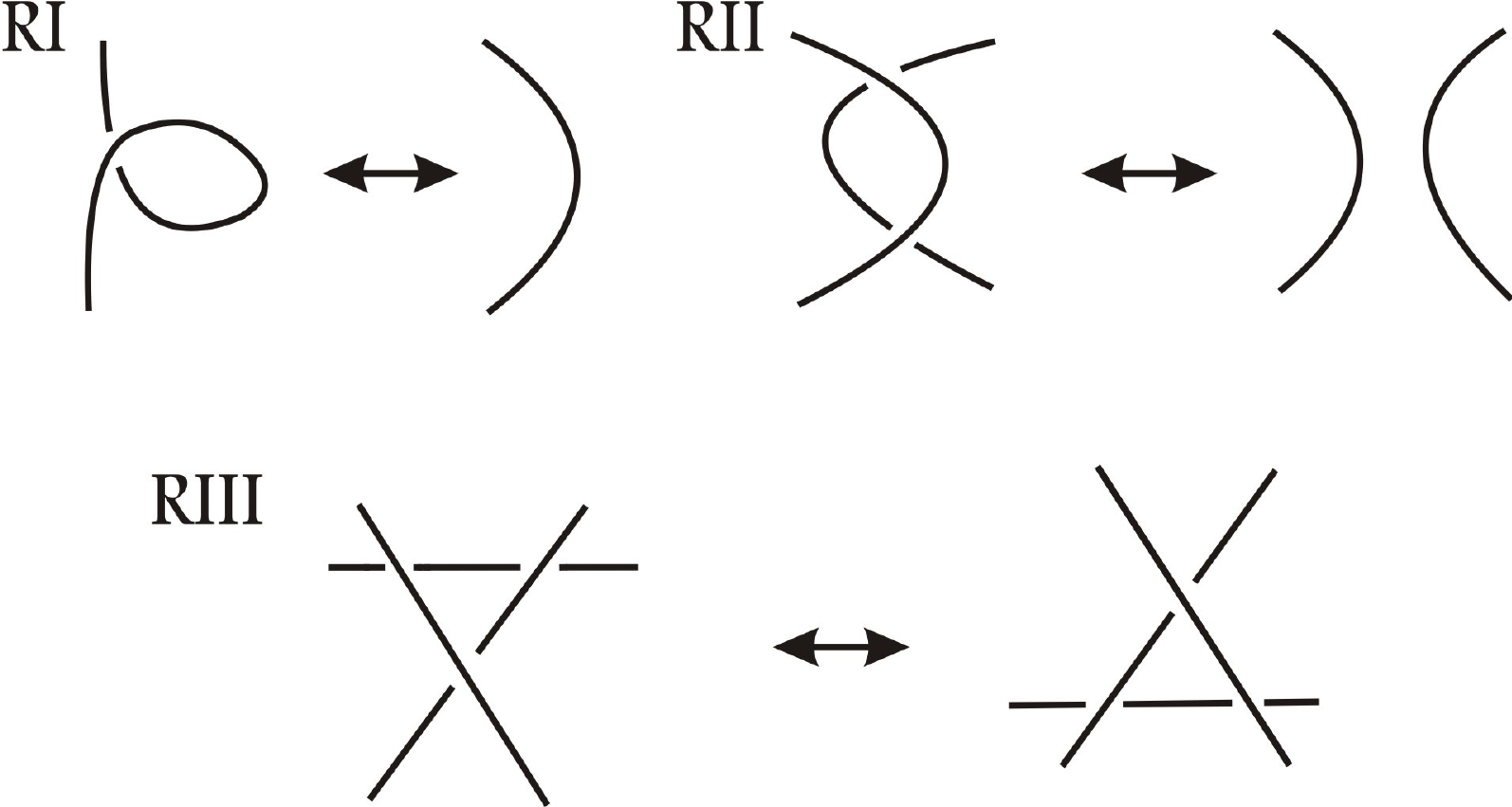}\label{cr}} \qquad 
\subfigure[Virtual Reidemeister moves.]{\includegraphics[scale=0.3]{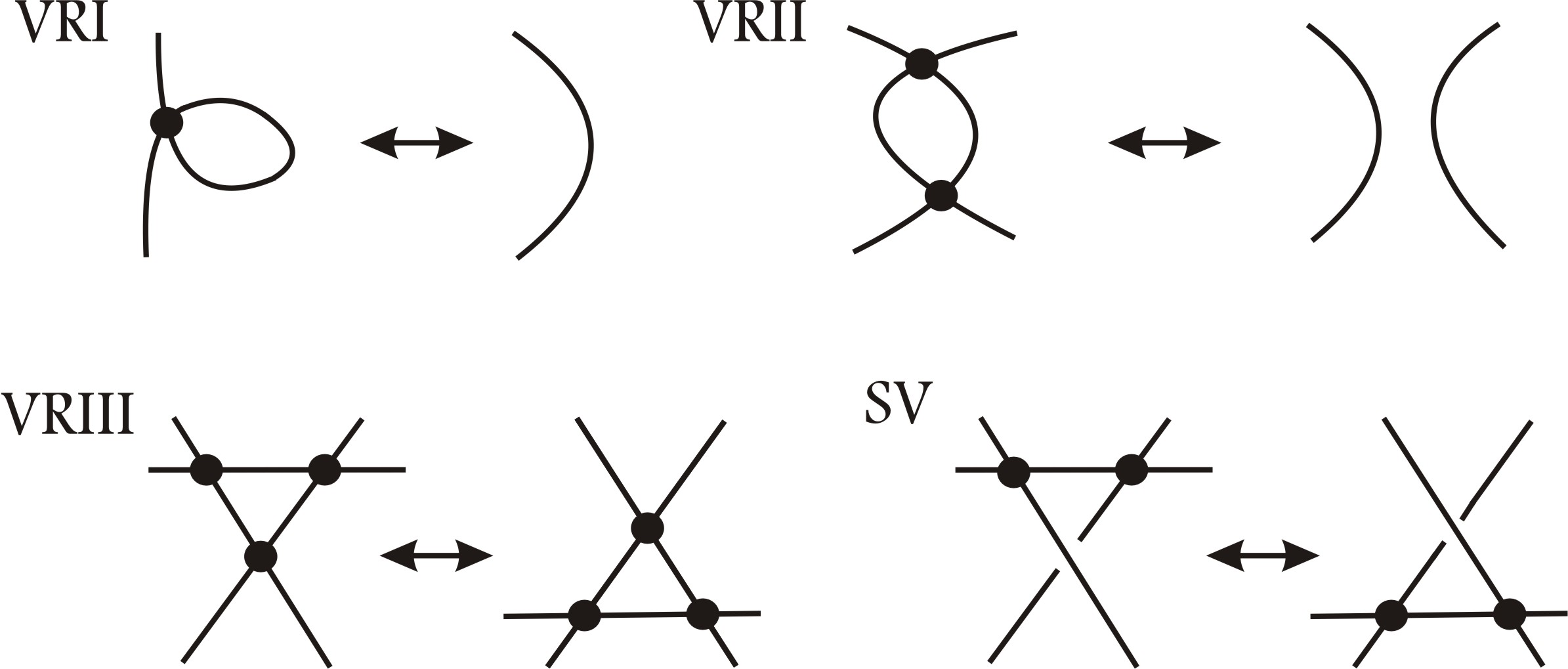} \label{vr} }
\caption{Generalized Reidemeister moves.}\label{rm} }
\end{figure} 

A \emph{detour move} is a  consequence of virtual Reidemeister moves which allows to move a segment of diagram containing only virtual crossings freely in the plane. We will consider a welded version of a detour move applying it to segments of diagrams containing only welded crossings.  While moving such a segment to different position in the plane, mark all the new crossings as welded crossings wherever it cuts across the knot diagram, see Fig.~\ref{DETOUR}. 

\begin{figure}[h] 
\centerline{\includegraphics[width=0.6\linewidth]{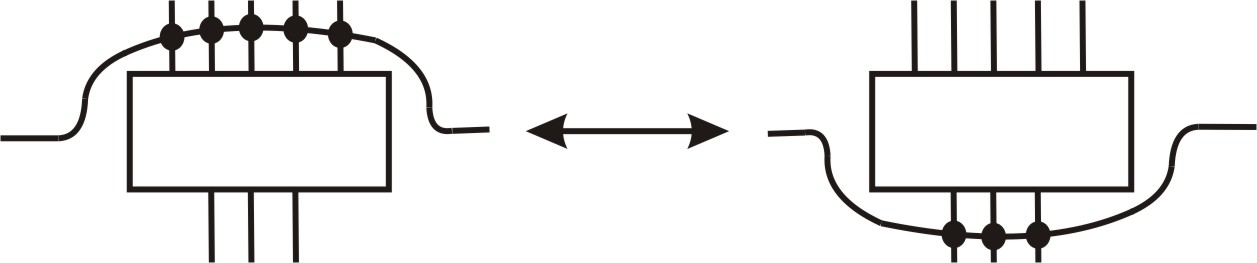}} 
\caption{Detour move.} \label{DETOUR}
\end{figure} 

There are two other diagrammatic moves $F_1$ and $F_2$ shown in Fig.~\ref{frb} not included in  generalized Reidemeister moves. It was shown in~\cite{kanenobu2001forbidden, nelson2001unknotting} that allowing moves $F_{1}$ and $F_{2}$ results in every virtual knot becoming equivalent to trivial knot. For this reason $F_1$ and $F_2$ are called \emph{forbidden moves}.
\begin{figure}[!ht] {\centering
\subfigure[$F_{1}$-move] {\includegraphics[scale=0.35]{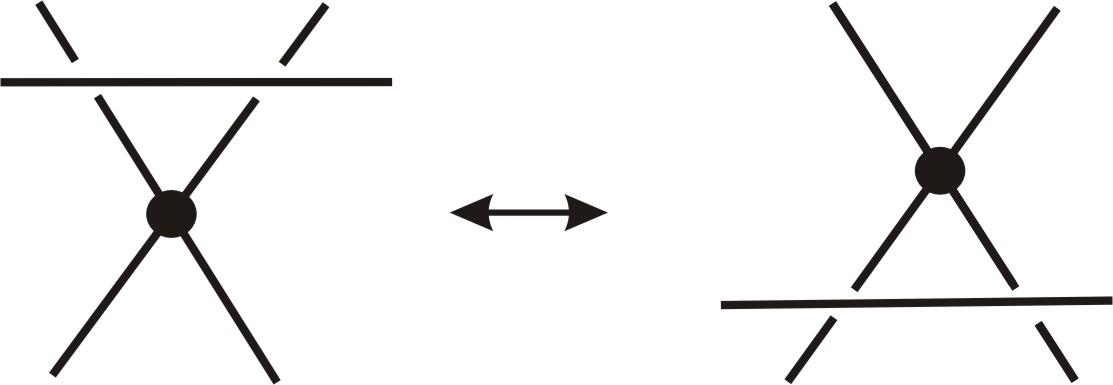}\label{frb1}} \qquad \qquad 
\subfigure[$F_{2}$-move]{\includegraphics[scale=0.35]{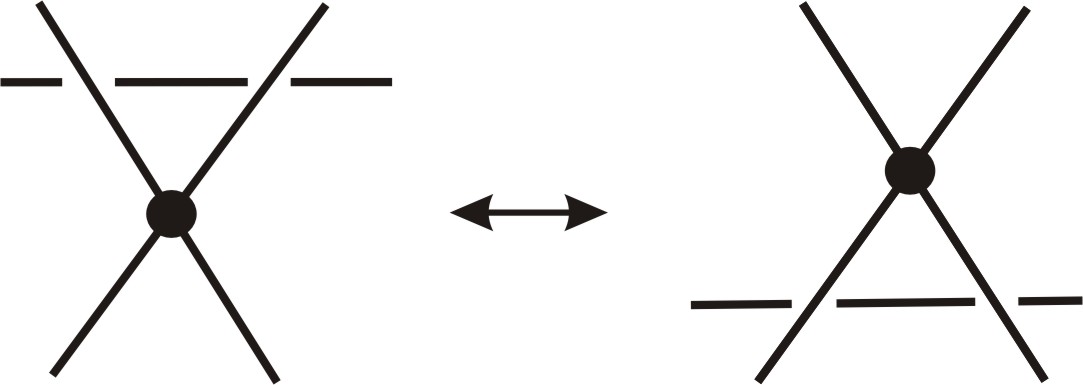} \label{frb2} }
\caption{Forbidden moves.}\label{frb} }\end{figure}
However, allowing only one of $F_1$ and $F_2$ move along with generalized Reidemeister moves results in non-trivial theory distinct from virtual knot theory. Virtual knots modulo forbidden move $F_1$ gives \emph{welded knots} similar to the ones introduced in~\cite{fenn1997braid}. Generalized Reidemeister moves along with $F_1$-move are called \emph{welded Reidemeister moves}. To distinguish the welded diagrams from virtual diagrams we will indicate welded crossings by a small disk as in the right in Fig.~\ref{fig1}. 

In the recent past, many knot invariants like knot group and knot quandle have been extended to virtual knots and welded knots. For example, the virtual version of Jones polynomial is an important polynomial invariant of virtual knots. It is known that Jones polynomial when considered under $F_1$-move is no longer an invariant hence indicating that the welded knot theory is not same as the virtual knot theory and therefore it is interesting to study the differences among them. Local moves play an important role in defining knot invariants, it was shown in \cite{kawauchi2012survey} that local moves like crossing change, $\Delta$-move and $\sharp$-move are unknotting operations for classical knots,  but not for virtual knots. However, in \cite{Sathohwelded} S.~Satoh  proved that crossing change, $\Delta$-move and $\sharp$-move are unknotting operations for welded knots hence remarking the difference of virtual and welded knots. An unknotting index for virtual links was investigated in~\cite{KPV}. 

It is still unknown whether there exists a classical nontrivial knot with unit Jones polynomial. However, examples of nontrivial virtual knots with unit Jones polynomial exists \cite{silver2004class} though none of them is yet known to be a classical knot. A local move called \emph{virtualization} and presented in Fig.~\ref{vc} was used for the construction of such virtual knots with unit Jones polynomial. 
\begin{figure}[!ht]
\centering 
\unitlength=0.4mm
\begin{picture}(140,30)(-10,5)
\thicklines
\qbezier(-20,10)(-20,10)(-10,10)
\qbezier(-20,30)(-20,30)(-10,30)
\qbezier(-10,10)(-10,10)(10,30)
\qbezier(-10,30)(-10,30)(-2,22) 
\qbezier(10,10)(10,10)(2,18)
\qbezier(20,10)(20,10)(10,10)
\qbezier(20,30)(20,30)(10,30)
\put(60,20){\makebox(0,0)[c,c]{$\Longleftrightarrow$}}
\qbezier(110,10)(110,10)(90,30)
\qbezier(110,30)(110,30)(90,10) 
\qbezier(110,10)(110,10)(130,30)
\qbezier(110,30)(110,30)(118,22) 
\qbezier(130,10)(130,10)(122,18)
\qbezier(150,10)(150,10)(130,30)
\qbezier(150,30)(150,30)(130,10) 
\put(100,20){\circle{4}}
\put(140,20){\circle{4}}
\end{picture}
\caption{Virtualization move.} \label{vc}
\end{figure}
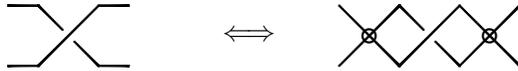
Remark that a virtualization move is not an unknotting operation for virtual knots. This fact follows from the observation by L.~Kauffman \cite{kauffman1999virtual} that bracket polynomial $\langle K \rangle$ and hence the Jones polynomial $f_K(A)$ of a virtual knot $K$ remains unchanged under virtualization move. Hence for a virtual knot $K$ with Jones polynomial $f_K(A) \neq 1$ it is not possible to convert $K$ into trivial knot using a sequence of virtualization moves and generalized Reidemeister moves. 

We study this local move in the category of welded knots establishing it as an unknotting operation and for the terminology call it \emph{twist move}, see Fig.~\ref{tm}.  The \emph{twist number} of a welded knot is defined as the minimal number of twist moves required to make a welded knot trivial. 
\begin{figure}[!ht]
\centering 
\unitlength=0.4mm
\begin{picture}(140,30)(-10,5)
\thicklines
\qbezier(-20,10)(-20,10)(-10,10)
\qbezier(-20,30)(-20,30)(-10,30)
\qbezier(-10,10)(-10,10)(10,30)
\qbezier(-10,30)(-10,30)(-2,22) 
\qbezier(10,10)(10,10)(2,18)
\qbezier(20,10)(20,10)(10,10)
\qbezier(20,30)(20,30)(10,30)
\put(60,20){\makebox(0,0)[c,c]{$\Longleftrightarrow$}}
\qbezier(110,10)(110,10)(90,30)
\qbezier(110,30)(110,30)(90,10) 
\qbezier(110,10)(110,10)(130,30)
\qbezier(110,30)(110,30)(118,22) 
\qbezier(130,10)(130,10)(122,18)
\qbezier(150,10)(150,10)(130,30)
\qbezier(150,30)(150,30)(130,10) 
\put(100,20){\circle*{4}}
\put(140,20){\circle*{4}}
\end{picture}
\caption{Twist move.} \label{tm}
\end{figure}
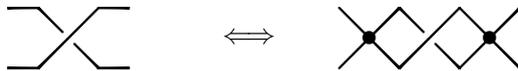
We observe that virtualization move fails to be unknotting operation for virtual knots. Further we study the relation of warping degree and welded unknotting number of a welded knot with the twist number. A lower bound on twist number is also established using quandle colorings. In the end we also discuss Gordian distance of welded knots by twist move and give some related results.  

The paper is organized as follows. In Section~2, we recall descending welded knot diagram, warping degree, quandle coloring and define welded unknotting number for welded knots. In Section~3, we prove that the twist move is an unknotting operation for welded knots, see Theorem~\ref{thm1}, and demonstrate lower and upper bounds of the unknotting twist number by welded unknotting number and warping degree, see Theorem~\ref{thm2}. In Section~4 we demonstrate that there is an infinite family of 2-bridge knots with unknotting twist number one, see Proposition~\ref{prop4.1}. 
In Section~5, we provide another lower bound on the twist number using Alexander quandle coloring. In the last Section~6 we consider the distance between welded knots by twist move, introduce corresponding  Gordian complex and study its properties, see Theorem~\ref{thm6.1}. 

\section{Descending diagrams and warping degree}

In this section, we recall the definition of descending welded knot diagram referring to~\cite{Sathohwelded} and warping degree for welded knots referring to~\cite{unknottingweldedLi2017}. 

We say that two welded knot diagrams $D$ and $D'$ are \emph{equivalent} if there exists a finite sequence of welded diagrams $D_{1} = D, D_{2}, \ldots, D_{n} = D'$ such that $D_{i+1}$ is obtained from $D_{i}$ by applying generalized Reidemeister moves or $F_{1}$, $i = 1, 2, \ldots, n-1$. A \emph{welded knot} is an equivalence class of welded knot diagrams. A welded knot is said to be \emph{trivial} if it is presented by a trivial diagram. It is easy to see~\cite[Lemma~2.3]{Sathohwelded} that any non-trivial welded knot $K$ has at least $3$ classical crossings in its diagram.  
A welded knot diagram $D$ is said to be \emph{descending} if there is a base point and an orientation of $D$ such
that walking along $D$ from the base point with respect to the orientation, we meet
the overcrossing for the first time and the undercrossing for the second time at
every classical crossing. Descending diagrams have the following important property.  

\begin{lemma}\cite[Prop.~2.2]{Sathohwelded} \label{l1}
Any descending diagram $D$ is related to the trivial diagram by 
a finite sequence of welded Reidemeister moves RI, VRI--VRIII, SV,  and  $F_1$.
\end{lemma}

In~\cite{a.shimizuwarping}, A. Shimizu introduced the concept of warping degree for classical knots. In~\cite{unknottingweldedLi2017}, Z.~Li, F.~Lei and J.~Wu extended the concept of warping degree to welded knots.   Let $D$ be an oriented welded knot diagram of a welded knot $K$ and $a$ be a non-crossing point on $D$.  Diagram $D_a$ is called \emph{based diagram} with based point $a$.  
The \emph{warping degree} of $D_a$, denoted by $d(D_a )$, is the number of classical crossings encounter first at under crossing point while starting from $a$ and traverse along the orientation of $D_a$.
The \emph{warping degree} of $D$, denoted by $d(D)$, is the minimal warping degree for all base points of $D$.
Let $-D$ be the inverse of $D$ obtained by reversing the orientation. Then \emph{warping degree of a knot} $K$ is defined as
$$
d(K) = \min \{d(D), d(-D) ~| ~D \in[K]\}.
$$

Further, switching a classical crossing to virtual crossing is an unknotting operation for virtual knots. The virtual unknotting number $u_v(K)$ is the minimum number of classical crossings needed to make virtual crossings to transform $K$ into the unknot. Similarly, the welded unknotting number can be defined for welded knots. Since a welded knot diagram with all welded crossings is a diagram of the trivial welded knot, virtual unknotting number is well defined for welded knots. 

\begin{definition}
{\rm 
\emph{The welded unknotting number $u_w (D)$ of a diagram} $D$ is the minimum number of classical crossings that must be changes to welded crossings so that the resulting diagram is a diagram of the trivial welded knot. \emph{The welded unknotting number $u_w (K)$ of a welded knot $K$} is defined as 
$$
u_w (K)= \min\{u_w (D) \, |  \, D \in [K]\}.
$$
}
\end{definition} 

Since $u_w (K)$ is minimum over all the diagrams presenting $K$, $u_w (K)$ is a welded knot invariant. Below we will use $u_w (K)$ to provide a lower bound on the twist number.
 \section{Twist number for welded knots}
 
 Let $D$ be a welded knot diagram and $c$ be a classical  crossing in $D$. Twist move is a local transformation which changes under crossing point and over crossing point of the crossing $c$ with each other as shown in Fig.~\ref{wc}, where $T_{1}$-move and $T_{2}$-move are indicated.  While performing this local transformation, sign of the crossing remains same, but two new welded crossings appear.  
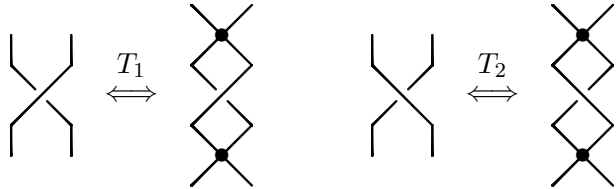
\begin{figure}[!ht]
\centering 
\unitlength=0.4mm
\begin{picture}(0,60)(0,-10)
\thicklines
\qbezier(-100,0)(-100,0)(-100,10)
\qbezier(-80,0)(-80,0)(-80,10)
\qbezier(-100,10)(-100,10)(-80,30)
\qbezier(-100,30)(-100,30)(-92,22) 
\qbezier(-80,10)(-80,10)(-88,18)
\qbezier(-100,30)(-100,30)(-100,40)
\qbezier(-80,30)(-80,30)(-80,40)
\put(-60,20){\makebox(0,0)[c,c]{$\Longleftrightarrow$}}
\put(-60,30){\makebox(0,0)[c,c]{$T_{1}$}}
\qbezier(-40,-10)(-40,-10)(-20,10)
\qbezier(-40,10)(-40,10)(-20,-10) 
\put(-30,0){\circle*{4}}
\qbezier(-40,10)(-40,10)(-20,30)
\qbezier(-40,30)(-40,30)(-32,22) 
\qbezier(-20,10)(-20,10)(-28,18)
\qbezier(-40,30)(-40,30)(-20,50)
\qbezier(-40,50)(-40,50)(-20,30) 
\put(-30,40){\circle*{4}}
\qbezier(20,0)(20,0)(20,10)
\qbezier(40,0)(40,0)(40,10)
\qbezier(20,30)(20,30)(40,10)
\qbezier(20,10)(20,10)(28,18) 
\qbezier(40,30)(40,30)(32,22)
\qbezier(20,30)(20,30)(20,40)
\qbezier(40,30)(40,30)(40,40)
\put(60,20){\makebox(0,0)[c,c]{$\Longleftrightarrow$}}
\put(60,30){\makebox(0,0)[c,c]{$T_{2}$}}
\qbezier(80,-10)(80,-10)(100,10)
\qbezier(80,10)(80,10)(100,-10) 
\put(90,0){\circle*{4}}
\qbezier(80,30)(80,30)(100,10)
\qbezier(80,10)(80,10)(88,18) 
\qbezier(100,30)(100,30)(92,22)
\qbezier(80,30)(80,30)(100,50)
\qbezier(80,50)(80,50)(100,30) 
\put(90,40){\circle*{4}}
\end{picture}
\caption{Twist moves.} \label{wc}
\end{figure}

Denote by $C(D)$ the set of all classical crossings of a welded knot diagram~$D$. 

\begin{proposition}
If $D'$ is a diagram obtained from a welded knot diagram $D$ by applying a twist move twice at the same crossing $c \in C(D)$, then $D'$ is equivalent to $D$.  
\end{proposition}
 
 \begin{proof}
Operate a twist move $T_1$ twice at $c$ as in Fig.~\ref{dcc1}, or a twist move $T_2$ twice at $c$ if it has another crossing type. 
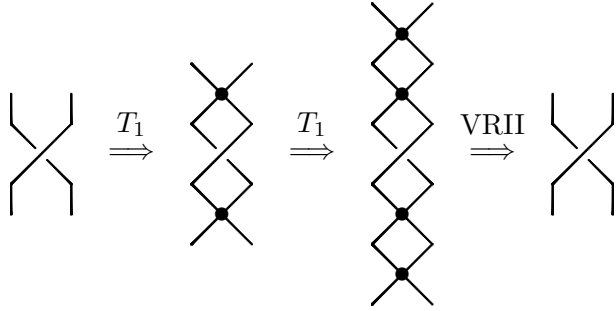
\begin{figure}[!ht]
\centering 
\unitlength=0.4mm
\begin{picture}(0,100)(0,-30)
\thicklines
\qbezier(-100,0)(-100,0)(-100,10)
\qbezier(-80,0)(-80,0)(-80,10)
\qbezier(-100,10)(-100,10)(-80,30)
\qbezier(-100,30)(-100,30)(-92,22) 
\qbezier(-80,10)(-80,10)(-88,18)
\qbezier(-100,30)(-100,30)(-100,40)
\qbezier(-80,30)(-80,30)(-80,40)
\put(-60,20){\makebox(0,0)[c,c]{$\Longrightarrow$}}
\put(-60,30){\makebox(0,0)[c,c]{$T_{1}$}}
\qbezier(-40,-10)(-40,-10)(-20,10)
\qbezier(-40,10)(-40,10)(-20,-10) 
\put(-30,0){\circle*{4}}
\qbezier(-40,10)(-40,10)(-20,30)
\qbezier(-40,30)(-40,30)(-32,22) 
\qbezier(-20,10)(-20,10)(-28,18)
\qbezier(-40,30)(-40,30)(-20,50)
\qbezier(-40,50)(-40,50)(-20,30) 
\put(-30,40){\circle*{4}}
\put(0,20){\makebox(0,0)[c,c]{$\Longrightarrow$}}
\put(0,30){\makebox(0,0)[c,c]{$T_{1}$}}
\qbezier(20,-30)(20,-30)(40,-10)
\qbezier(20,-10)(20,-10)(40,-30)
\put(30,-20){\circle*{4}}
\qbezier(20,-10)(20,-10)(40,10)
\qbezier(20,10)(20,10)(40,-10)
\put(30,0){\circle*{4}}
\qbezier(20,10)(20,10)(40,30)
\qbezier(20,30)(20,30)(28,22) 
\qbezier(40,10)(40,10)(32,18)
\qbezier(20,30)(20,30)(40,50)
\qbezier(20,50)(20,50)(40,30)
\put(30,40){\circle*{4}}
\qbezier(20,50)(20,50)(40,70)
\qbezier(20,70)(20,70)(40,50)
\put(30,60){\circle*{4}}
\put(60,20){\makebox(0,0)[c,c]{$\Longrightarrow$}}
\put(60,30){\makebox(0,0)[c,c]{VRII}}
\qbezier(80,10)(80,10)(100,30)
\qbezier(80,30)(80,30)(88,22) 
\qbezier(100,10)(100,10)(92,18)
\qbezier(80,0)(80,0)(80,10) 
\qbezier(100,0)(100,0)(100,10)
\qbezier(80,30)(80,30)(80,40) 
\qbezier(100,30)(100,30)(100,40)
\end{picture}
\caption{Performing twist move $T_{1}$ twice.} \label{dcc1}
\end{figure}
Also, operate twist moves $T_1$ and $T_2$ as in Fig.~\ref{dcc2} or $T_2$ and $T_1$ is the case of another crossing type at $c$.  Let $D'$ be the resulting diagram obtained from $D$ by performing twist moves twice at crossing $c$. 
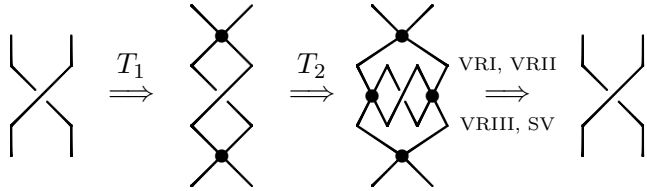
\begin{figure}[!ht]
\centering 
\unitlength=0.4mm
\begin{picture}(0,60)(0,-10)
\thicklines
\qbezier(-100,0)(-100,0)(-100,10)
\qbezier(-80,0)(-80,0)(-80,10)
\qbezier(-100,10)(-100,10)(-80,30)
\qbezier(-100,30)(-100,30)(-92,22) 
\qbezier(-80,10)(-80,10)(-88,18)
\qbezier(-100,30)(-100,30)(-100,40)
\qbezier(-80,30)(-80,30)(-80,40)
\put(-60,20){\makebox(0,0)[c,c]{$\Longrightarrow$}}
\put(-60,30){\makebox(0,0)[c,c]{$T_{1}$}}
\qbezier(-40,-10)(-40,-10)(-20,10)
\qbezier(-40,10)(-40,10)(-20,-10) 
\put(-30,0){\circle*{4}}
\qbezier(-40,10)(-40,10)(-20,30)
\qbezier(-40,30)(-40,30)(-32,22) 
\qbezier(-20,10)(-20,10)(-28,18)
\qbezier(-40,30)(-40,30)(-20,50)
\qbezier(-40,50)(-40,50)(-20,30) 
\put(-30,40){\circle*{4}}
\put(0,20){\makebox(0,0)[c,c]{$\Longrightarrow$}}
\put(0,30){\makebox(0,0)[c,c]{$T_{2}$}}
\qbezier(20,-10)(20,-10)(30,0)
\qbezier(30,0)(30,0)(40,-10)
\qbezier(15,10)(30,0)(30,0)
\qbezier(45,10)(30,0)(30,0)
\put(30,0){\circle*{4}}
\qbezier(15,10)(15,10)(25,30)
\qbezier(15,30)(15,30)(25,10)
\put(20,20){\circle*{4}}
\qbezier(25,10)(25,10)(35,30)
\qbezier(25,30)(25,30)(29,22) 
\qbezier(35,10)(35,10)(31,18)
\qbezier(35,10)(35,10)(45,30)
\qbezier(35,30)(35,30)(45,10)
\put(40,20){\circle*{4}}
\qbezier(20,50)(20,50)(30,40)
\qbezier(40,50)(40,50)(30,40)
\qbezier(15,30)(15,30)(30,40)
\qbezier(45,30)(45,30)(30,40)
\put(30,40){\circle*{4}}
\put(65,20){\makebox(0,0)[c,c]{$\Longrightarrow$}}
\put(65,30){\makebox(0,0)[c,c]{\tiny{VRI, VRII}}} 
\put(65,10){\makebox(0,0)[c,c]{\tiny{VRIII, SV}}}
\qbezier(90,10)(90,10)(110,30)
\qbezier(90,30)(90,30)(98,22) 
\qbezier(110,10)(110,10)(102,18)
\qbezier(90,0)(90,0)(90,10) 
\qbezier(110,0)(110,0)(110,10)
\qbezier(90,30)(90,30)(90,40) 
\qbezier(110,30)(110,30)(110,40)
\end{picture}
\caption{Performing twist moves $T_{1}$ and $T_{2}$.} \label{dcc2}
\end{figure}
It is shown in Fig.~\ref{dcc1} and Fig.~\ref{dcc2} that $D'$ is equivalent to $D$ by welded Reidemeister moves.  Analogous equivalence hold in other cases. 
\end{proof}

\begin{theorem} \label{thm1}
 Twist move is an unknotting operation for welded knots.
 \end{theorem}
 
 \begin{proof} Any welded knot diagram can be turned into trivial welded knot diagram by using generalized Reidemeister moves along with both forbidden moves $F_{1}$ and $F_{2}$. Recall that $F_{1}$ is one of welded Reidemeister moves. 
 Now we will show that forbidden move $F_2$ can be realized using twist moves along with welded Reidemeister moves.
Indeed, we demonstrate in Fig.~\ref{figF2} that $F_2$-move as shown in Fig.~\ref{frb2} can be realized by twist moves along with welded Reidemeister moves.
\end{proof}

 \begin{figure}[ht]{\vspace{-.1cm}
 \centering
\subfigure{\includegraphics[scale=0.51]{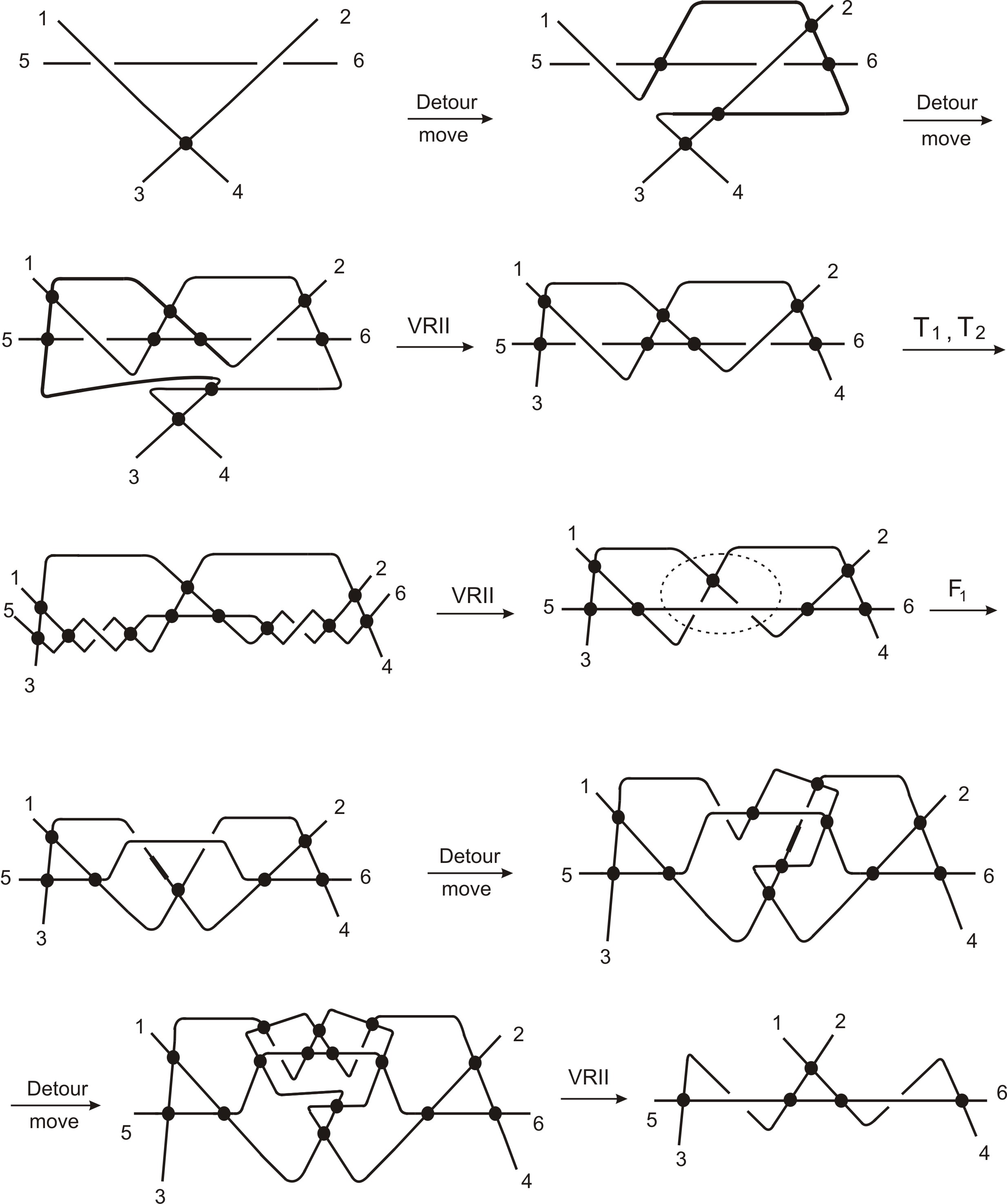}} 
 \subfigure{\includegraphics[scale=0.51]{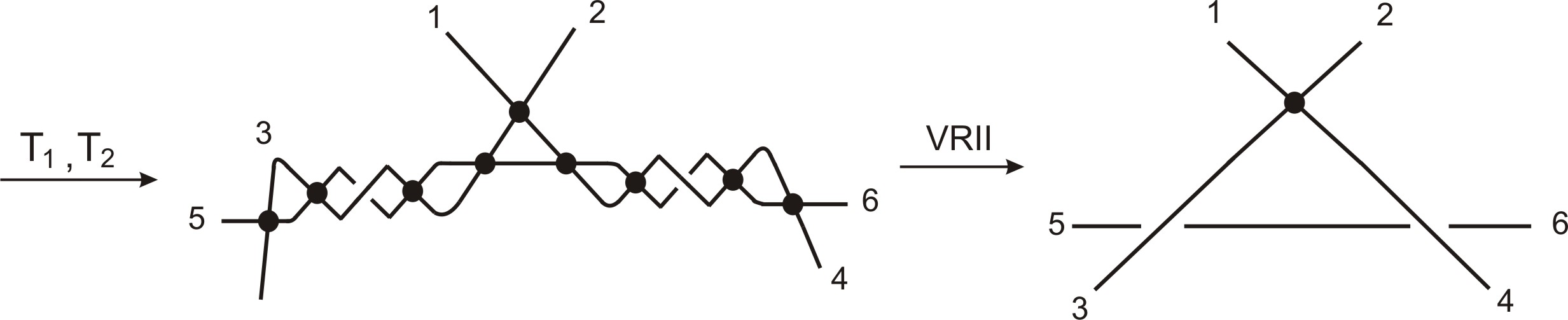} }
\caption{Realization of $F_2$-move.}\label{figF2} }
\end{figure} 

Using twist moves, we define an invariant for welded knot. 
Let $D$ be a welded knot diagram. We define \emph{unknotting twist number} $ut(D)$ of $D$ as the minimum number of twist moves that are required to change $D$ into a diagram of the trivial welded knot. 

\begin{definition}
{\rm 
\emph{The unknotting twist number} $ut(K)$ of a welded knot $K$ is defined as the minimum number of twist moves required, taken over all welded knot diagrams representing $K$, to convert $K$ into the trivial welded knot, {\it i.e,}
$$
ut(K)= \min\{ut(D)~|~D\in[K]\}.
$$
}
 \end{definition}
 
Since $ut(K)$ is minimum over all diagrams presenting $K$, it is a welded knot invariant. If $K$ is a trivial welded knot, then $ut(K)=0$. Otherwise, $ut(K)\geq 1$. 
 
\begin{example}\label{exmpl1} 
{\rm 
Let $B_{1}$ be a welded knot presented by a diagram in Fig.~\ref{utk}.  
\begin{figure}[!ht]
\centering 
\unitlength=0.4mm
\begin{picture}(0,140)(-20,-70)
\thicklines
\qbezier(-130,-50)(-130,-50)(-110,-30)
\qbezier(-110,-50)(-110,-50)(-130,-30) 
\put(-120,-40){\circle*{4}}
\qbezier(-130,-30)(-130,-30)(-110,-10)
\qbezier(-130,-10)(-130,-10)(-122,-18) 
\qbezier(-110,-30)(-110,-30)(-118,-22)
\qbezier(-130,10)(-130,10)(-110,-10)
\qbezier(-130,-10)(-130,-10)(-110,10) 
\put(-120,0){\circle*{4}}
\qbezier(-130,30)(-130,30)(-110,10)
\qbezier(-130,10)(-130,10)(-122,18) 
\qbezier(-110,30)(-110,30)(-118,22)
\qbezier(-130,50)(-130,50)(-110,30)
\qbezier(-130,30)(-130,30)(-122,38) 
\qbezier(-110,50)(-110,50)(-118,42)
\qbezier(-130,50)(-135,50)(-135,45)
\qbezier(-135,45)(-135,45)(-135,-45)
\qbezier(-135,-45)(-135,-50)(-130,-50)
\qbezier(-110,50)(-105,50)(-105,45)
\qbezier(-105,45)(-105,45)(-105,-45)
\qbezier(-105,-45)(-105,-50)(-110,-50)
\put(-120,-60){\makebox(0,0)[c,c]{$B_{1}$}}
\put(-90,0){\makebox(0,0)[c,c]{$\Longrightarrow$}}
\put(-90,10){\makebox(0,0)[c,c]{$T_{2}$}}
\qbezier(-70,-70)(-70,-70)(-50,-50)
\qbezier(-70,-50)(-70,-50)(-50,-70) 
\put(-60,-60){\circle*{4}}
\qbezier(-70,-50)(-70,-50)(-50,-30)
\qbezier(-70,-30)(-70,-30)(-62,-38) 
\qbezier(-50,-50)(-50,-50)(-58,-42) 
\qbezier(-70,-30)(-70,-30)(-50,-10)
\qbezier(-70,-10)(-70,-10)(-50,-30) 
\put(-60,-20){\circle*{4}}
\qbezier(-70,-10)(-70,-10)(-50,10)
\qbezier(-70,10)(-70,10)(-50,-10) 
\put(-60,0){\circle*{4}}
\qbezier(-70,30)(-70,30)(-50,10)
\qbezier(-70,10)(-70,10)(-62,18) 
\qbezier(-50,30)(-50,30)(-58,22) 
\qbezier(-70,30)(-70,30)(-50,50)
\qbezier(-70,50)(-70,50)(-50,30) 
\put(-60,40){\circle*{4}}
\qbezier(-70,70)(-70,70)(-50,50)
\qbezier(-70,50)(-70,50)(-62,58) 
\qbezier(-50,70)(-50,70)(-58,62)
\qbezier(-70,70)(-75,70)(-75,65)
\qbezier(-75,65)(-75,65)(-75,-65) 
\qbezier(-75,-65)(-75,-70)(-70,-70)
\qbezier(-50,70)(-45,70)(-45,65)
\qbezier(-45,65)(-45,65)(-45,-65) 
\qbezier(-45,-65)(-45,-70)(-50,-70)
\put(-30,0){\makebox(0,0)[c,c]{$\Longrightarrow$}}
\put(-30,10){\makebox(0,0)[c,c]{VRII}}
\qbezier(-10,-50)(-10,-50)(10,-30)
\qbezier(-10,-30)(-10,-30)(10,-50)
\put(0,-40){\circle*{4}}
\qbezier(-10,-30)(-10,-30)(10,-10)
\qbezier(-10,-10)(-10,-10)(-2,-18)
\qbezier(10,-30)(10,-30)(2,-22)
\qbezier(-10,10)(-10,10)(10,-10)
\qbezier(-10,-10)(-10,-10)(-2,-2)
\qbezier(2,2)(10,10)(10,10)
\qbezier(-10,10)(-10,10)(10,30)
\qbezier(-10,30)(-10,30)(10,10) 
\put(0,20){\circle*{4}}
\qbezier(-10,50)(-10,50)(10,30)
\qbezier(-10,30)(-10,30)(-2,38)
\qbezier(10,50)(10,50)(2,42)
\qbezier(-10,50)(-15,50)(-15,45)
\qbezier(-15,45)(-15,45)(-15,-45)
\qbezier(-15,-45)(-15,-50)(-10,-50)
\qbezier(10,50)(15,50)(15,45)
\qbezier(15,45)(15,45)(15,-45)
\qbezier(15,-45)(15,-50)(10,-50)
\put(30,0){\makebox(0,0)[c,c]{$\Longrightarrow$}}
\put(30,10){\makebox(0,0)[c,c]{RII}}
\qbezier(50,-30)(50,-30)(70,-10)
\qbezier(50,-10)(50,-10)(70,-30) 
\put(60,-20){\circle*{4}}
\qbezier(50,-10)(50,-10)(70,10)
\qbezier(50,10)(50,10)(70,-10) 
\put(60,0){\circle*{4}}
\qbezier(50,30)(50,30)(70,10)
\qbezier(50,10)(50,10)(58,18)
\qbezier(70,30)(70,30)(62,22)
\qbezier(50,30)(45,30)(45,25)
\qbezier(45,25)(45,25)(45,-25)
\qbezier(45,-25)(45,-30)(50,-30)
\qbezier(70,30)(75,30)(75,25)
\qbezier(75,25)(75,25)(75,-25)
\qbezier(75,-25)(75,-30)(70,-30)
\put(90,0){\makebox(0,0)[c,c]{$\Longrightarrow$}}
\put(90,10){\makebox(0,0)[c,c]{VRII}}
\put(90,-10){\makebox(0,0)[c,c]{RI}}
\qbezier(110,10)(105,10)(105,0)
\qbezier(105,0)(105,-10)(110,-10)
\qbezier(110,-10)(115,-10)(115,0)
\qbezier(115,0)(115,10)(110,10)
\end{picture}
\caption{Diagram of welded knot $B_{1}$.} \label{utk}
\end{figure}
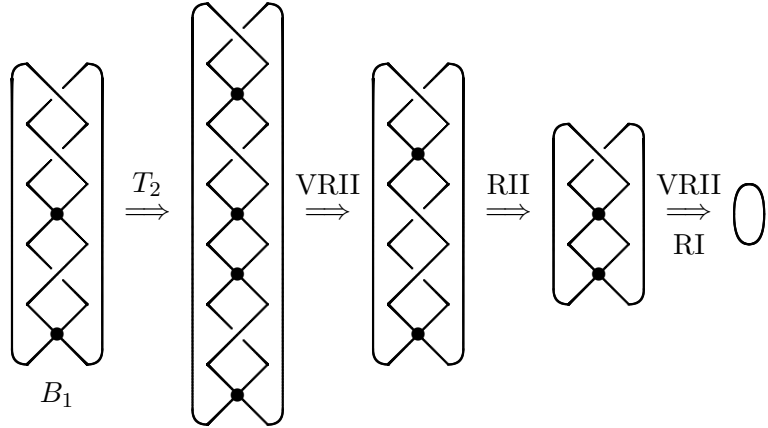
Since by \cite{unknottingweldedLi2017} $B_{1}$ is a non-trivial welded knot, we have $ut(B_{1})\geq 1$. 
As shown in Fig.~\ref{utk}, the diagram of $B_{1}$ can be deformed into a trivial welded knot diagram by applying one twist move along with welded Reidemeister moves. Therefore,  $ut(B_{1})\leq 1$. Hence $ut(B_{1})=1$.
}
\end{example}

The following result  provides lower bound and upper bound in terms of welded unknotting number and warping degree of the welded knot. 

\begin{theorem} \label{lwrbd} \label{thm2}
If $K$ is a welded knot, then $$\frac{1}{2} u_w(K)\leq ut(K)\leq d(K).$$
 \end{theorem}
 
 \begin{proof}
 Let $D$ be a diagram of a welded knot $K$. It is easy to see that  minimum number of twist moves required to convert $D$ into a descending diagram is $\min\{d(D),~d(-D)\}$. Since a descending diagram is a trivial welded knot diagram, 
$$
ut(D)\leq \min\{d(D),~d(-D)\}.
$$ 
Thus $ut(K)\leq \min\{d(D),~d(-D)\}$, for any diagram $D$ presenting $K$. Hence $ut(K)\leq d(K)$.
 
Let $D'$ be a diagram of $K$ that realizes $ut (K)$. At every crossing $c$ where a twist move is required to unknot $D$, we produce a RII-move and then change two corner crossings to welded crossings (C - W) as presented in Fig.~\ref{lwbd}.  
 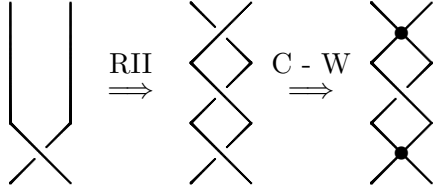
\begin{figure}[!ht]
\centering 
\unitlength=0.4mm
\begin{picture}(0,60)(0,-30)
\thicklines  
\qbezier(-70,-10)(-70,-10)(-50,-30)
\qbezier(-70,-30)(-70,-30)(-62,-22) 
\qbezier(-50,-10)(-50,-10)(-58,-18) 
\qbezier(-70,-10)(-70,-10)(-70,30)
\qbezier(-50,-10)(-50,-10)(-50,30)  
\put(-30,0){\makebox(0,0)[c,c]{$\Longrightarrow$}}
\put(-30,10){\makebox(0,0)[c,c]{RII}}
\qbezier(-10,-10)(-10,-10)(10,-30)
\qbezier(-10,-30)(-10,-30)(-2,-22)
\qbezier(10,-10)(10,-10)(2,-18)
\qbezier(-10,10)(-10,10)(10,-10)
\qbezier(-10,-10)(-10,-10)(-2,-2)
\qbezier(2,2)(10,10)(10,10)
\qbezier(-10,10)(-10,10)(10,30)
\qbezier(-10,30)(-10,30)(-2,22)
\qbezier(10,10)(10,10)(2,18)
\put(30,0){\makebox(0,0)[c,c]{$\Longrightarrow$}}
\put(30,10){\makebox(0,0)[c,c]{C - W}}
\qbezier(50,-30)(50,-30)(70,-10)
\qbezier(50,-10)(50,-10)(70,-30) 
\put(60,-20){\circle*{4}}
\qbezier(50,10)(50,10)(70,-10)
\qbezier(50,-10)(50,-10)(58,-2)
\qbezier(70,10)(70,10)(62,2)
\qbezier(50,10)(50,10)(70,30)
\qbezier(50,30)(50,30)(70,10) 
\put(60,20){\circle*{4}}
\end{picture}
\caption{A twist move as a composition of RII and changing classical crossings to welded.} \label{lwbd} 
\end{figure}
 Clearly,  as a result  we obtain the same diagram
as would have been obtained by simply applying a twist move on $c$. Repeat
this process at every crossing at which a twist move is required to unknot $D$. At the end, we would have the same diagram as would be obtained by applying twist moves at all those crossings. Thus $u_{w}(K) \leq 2 ut(K).$
 \end{proof}

\section{Infinite family of 2-bridge knots with unknotting twist number one} 

In this section we will demonstrate that there exists an infinite family of welded knots with unknotting twist number one.  
Namely, this a family of classical knots considered up to welded equivalence. 
Let ${\bf b}(2n+\frac{1}{2})$ be a two-bridge knot with the rational parameter $2n + \frac{1}{2}$, for integer $n \geq 1$.  

\begin{proposition} \label{prop4.1} 
For any integer $n \geq 1$ we have $ut({\bf b(}2n + \frac12) )=1$. 
\end{proposition} 

\begin{proof}
Let $D_{2n}$ denotes a diagram of the two-bridge knot ${\bf b}(2n+\frac{1}{2})$ presented in Fig.~\ref{2B1}. Apply a twist-move $T_2$ in $D_{2n}$ so that as a result we obtain welded knot $D'_{2n}$ as shown in Fig.~\ref{2B1}. $D'_{2n}$ can be transformed into welded knot denoted by $D''_{2n}$ using welded Reidemeister moves, see Fig.~\ref{2B2}. Now the welded knot $D''_{2n}$ is equivalent via welded isotopy to the welded knot $D'_{2n-2}$ as can be seen in the Fig.~\ref{2B3}. Therefore we can convert $D'_{2n}$ into $D'_{2n-2}$ using welded isotopy, in particular by repeating the number of steps if required we obtain $D'_0$. As can be seen from Fig.~\ref{2B4}, $D'_0$ is equivalent to trivial welded knot by $RI$ and $VRI$-moves, hence $ut({\bf b(}2n + \frac12) ) \leq 1$. 
Recall, that if $K$ and $K'$ are classical knots whose diagrams are welded equivalent, then they are isotopic, see for example~\cite[Theorem~2.1]{Winter}. Therefore, $ut({\bf b(}2n + \frac12) ) \geq 1$ and the statement holds.
\begin{figure}[h] 
\centerline{\includegraphics[width=0.5\linewidth]{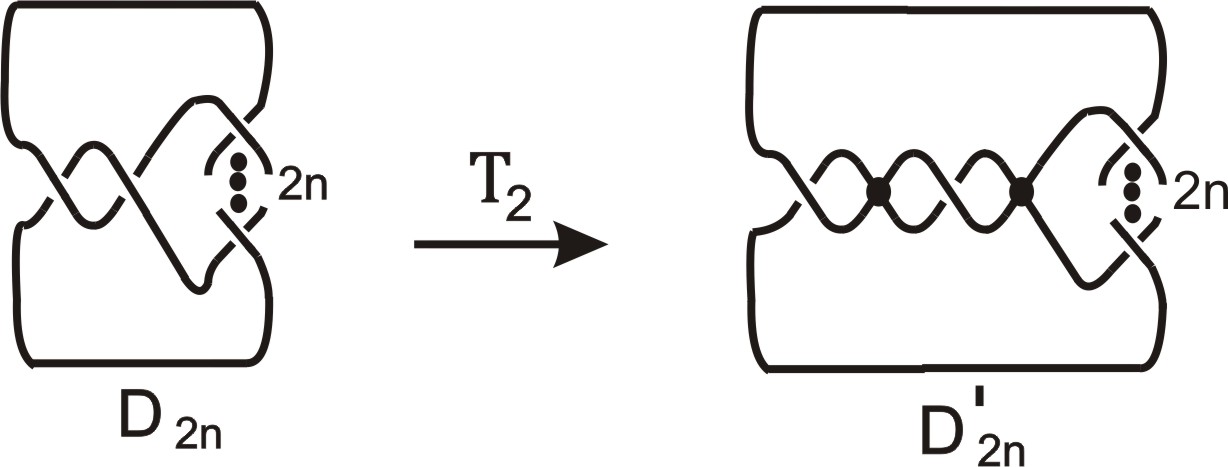}} 
\caption{$T_2$-move in two-bridge knot $D_{2n} = {\bf b}(2n+\frac{1}{2})$.} \label{2B1}
\end{figure}
\begin{figure}[h] 
\centerline{\includegraphics[width=0.8\linewidth]{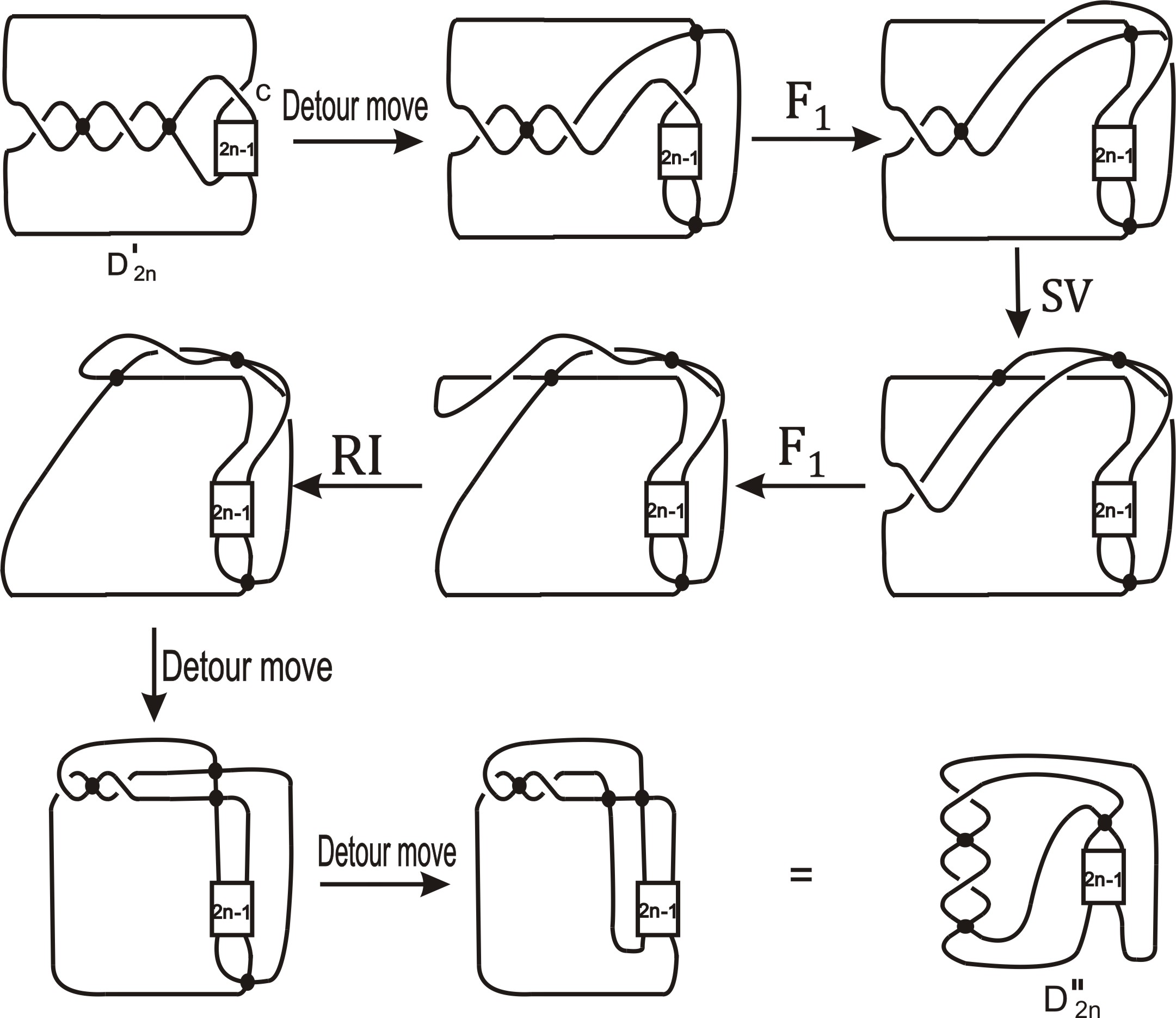}} 
\caption{Welded isotopy between $D'_{2n}$ and $D''_{2n}$.} \label{2B2}
\end{figure} 
\begin{figure}[h] 
\centerline{\includegraphics[width=0.8\linewidth]{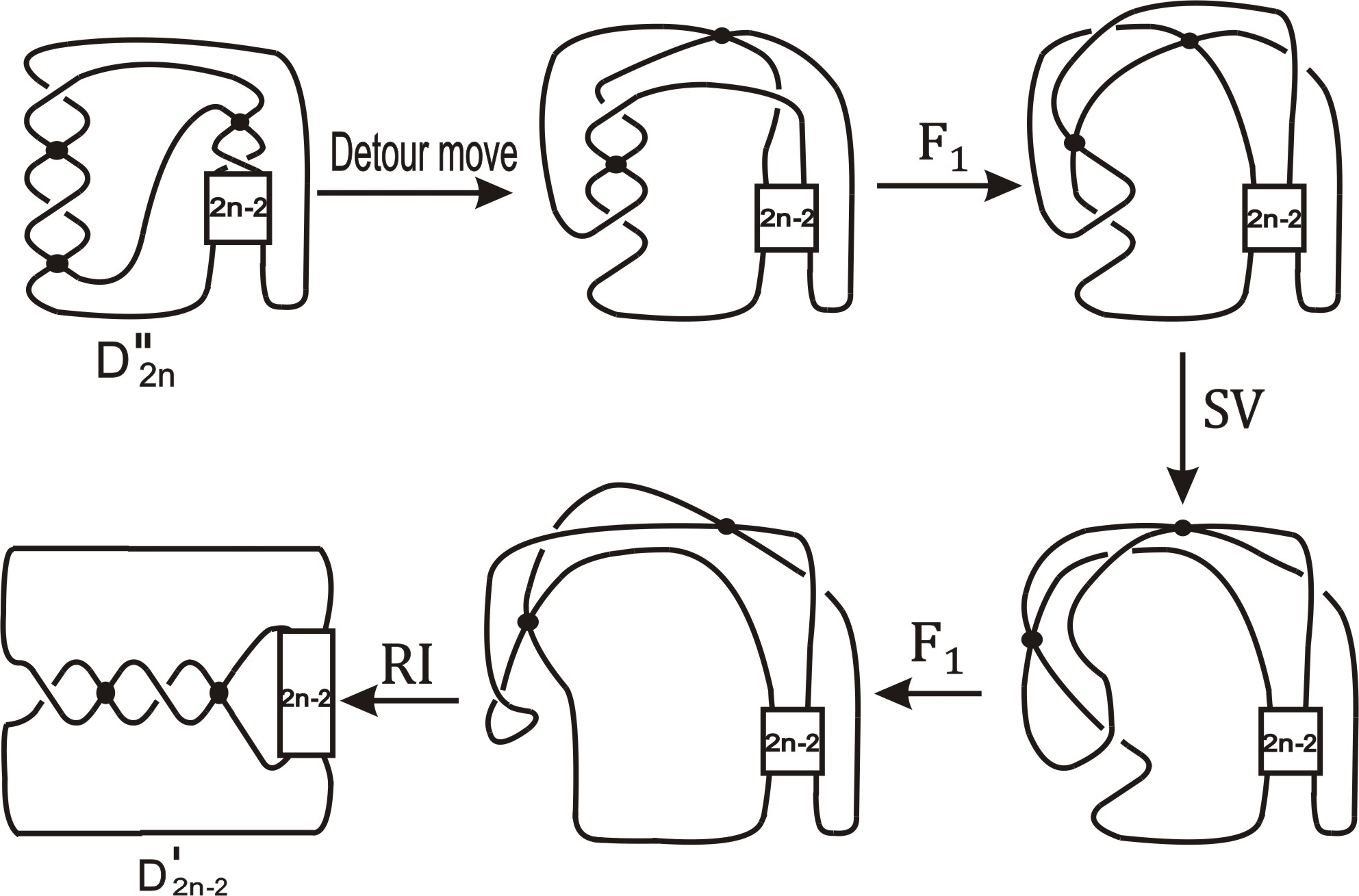}} 
\caption{Welded isotopy between $D''_{2n}$ and $D'_{2n-2}$.} \label{2B3}
\end{figure}
\begin{figure}[h] 
\centerline{\includegraphics[width=0.5\linewidth]{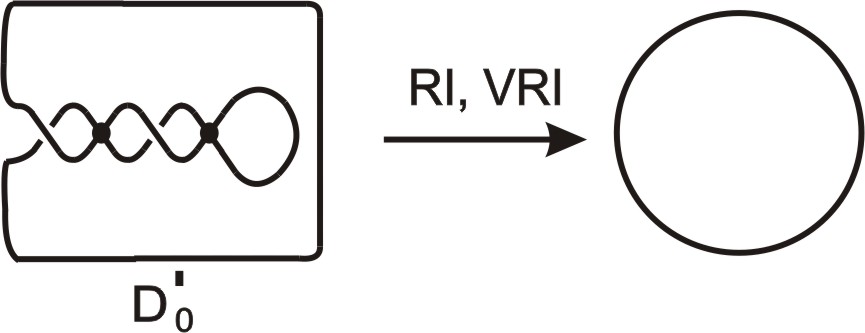}} 
\caption{$D'_{0}$ equivalent to trivial knot.} \label{2B4}
\end{figure}    
\end{proof}

\section{Quandle coloring and unknotting twist number}

An algebraic system known as quandle was  introduced independently  in  \cite{joyce1982classifying} and \cite{matveev1982distributive}.  

\begin{definition}
{\rm 
A \emph{quandle} is a non-empty set $X$ with a binary operation $*:X\times X \rightarrow X$ satisfying the following axioms:
\begin{enumerate}
\item For all $x\in X,$ $x*x=x.$
\item For all $y\in X$, the map $\phi_{y}:X\rightarrow X$ defined by $\phi_y(x)=x*y$ is invertible.
\item For all $x,y,z \in X$, $(x*y)*z=(x*z)*(y*z).$
\end{enumerate} 
}
\end{definition}
 
A \emph{dihedral quandle} $(R_n;*)$ of order $n$ is the ring $\mathbb{Z}_n (= \mathbb{Z}/n\mathbb{Z})$ with an operation defined by $x*y = 2y-x$.
An \emph{Alexander quandle} is a module over the ring  $\Lambda_p=\mathbb{Z}_{p}[t^{\pm 1}]$ of Laurent polynomials, considered with operation  $x*y=tx+(1-t)y$.   
Let  $f(t)$ be an irreducible polynomial of degree $d$ in $\Lambda_p$ with $f(t)\neq t, t-1$, then $F_q= \Lambda_{p} /f(t)$, where $q = p^{d}$ is a finite field. Elements of $F_{q}$ are presented by remainders of polynomials when divided by $f(t)$ mod $p$, so, indeed, $F_{q}$ consists of $p^{d}$ elements.   A quandle $(F_{q}, *)$ is a $\Lambda_{p}$-module with the operation $x*y = t x + (1-t) y$. 

Now we recall the definition of quandle coloring for welded knots~\cite{unknottingweldedLi2017}. 
Let $X$ be a fixed quandle and $D$ be an oriented welded knot or link diagram.
Let $A(D)$ be the set of arcs, where an arc is a piece of a curve each of whose endpoints is an undercrossing. The normals (normal vectors) are given in such a way that the ordered pair (tangent, normal) agrees with the orientation of the plane, see Fig.~\ref{qrel}.
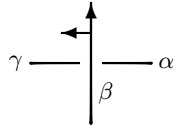
\begin{figure}[!ht]
\centering 
\unitlength=0.4mm
\begin{picture}(0,40)(0,-20)
\thicklines
\qbezier(0,-20)(0,-20)(0,20)
\qbezier(-20,0)(-20,0)(-4,0)
\qbezier(20,0)(20,0)(4,0)
\put(0,10){\vector(0,1){10}}
\put(0,10){\vector(-1,0){10}}
\put(25,0){\makebox(0,0)[c,c]{\footnotesize $\alpha$}}
\put(5,-10){\makebox(0,0)[c,c]{\footnotesize $\beta$}}
\put(-25,0){\makebox(0,0)[c,c]{\footnotesize $\gamma$}}
\end{picture}
\caption{Arcs incident to a crossing.} \label{qrel}
\end{figure}

A \emph{quandle coloring} $\mathcal{C}$ is a map $\mathcal{C}: A(D) \rightarrow X$ satisfying the following rule at every classical crossing. If the normal to the over-arc $\beta$ points from the arc $\alpha$ to $\gamma$ as in Fig.~\ref{qrel}, then $\mathcal{C}(\alpha)*\mathcal{C}(\beta) = \mathcal{C}(\gamma)$.  Let $\operatorname{Col}_X (D)$ denote the set of colorings of a welded knot diagram $D$ of a welded knot $K$ by a quandle $X$. Then the cardinality $\abs{\operatorname{Col}_X (D)}$ is a welded knot invariant. For more detail see \cite{fenn1992racks, unknottingweldedLi2017}.

From now we consider the quandle $X = F_{q}$. For a welded knot diagram $D$ the set $\operatorname{Col}_{F_{q}} (D)$ forms a linear space over $F_{q}$, whose generators are all elements in $A (D)$ and relations are given for each classical crossings as follows. 
Let $\alpha$ and $\beta$ be the under-passing arcs and  $\gamma$  be the over-passing arc at classical crossing $c$, such that the normal orientation of $\gamma$ points from $\alpha$ to $\beta$. Then $\beta = t \alpha +(1-t) \gamma$. 
Denoting the set of such relations by $R(D)$,  we have a presentation $\operatorname{Col}_{F_q} (D)\cong \langle A(D) \, | \, R(D)\rangle_{F_q}$.  
 
Similar to~\cite{unknottingweldedLi2017} we consider a presentation of $\operatorname{Col}_{F_{q}}(D)$ taking semi-arcs of the diagram as generators: 
 $\operatorname{Col}_{F_{q}}(D)\cong \langle SA(D) \, | \, SR(D)\rangle_{F_{q}}$, where:
\begin{enumerate}
\item $SA (D)$ denote the set of semi-arcs, where a semi-arc is a piece of curve whose endpoints are classical crossing points.
\item Let $\alpha, \beta, \gamma$ and $\delta$ be respectively the two under-crossing semi-arcs and two over-crossing semi-arcs at a classical crossing as presented in Fig.~\ref{1ap26}.  The normal orientation of $\gamma$ points from $\alpha$ to $\beta$ and that of $\alpha$ points from $\gamma$ to $\delta$. Then relations are  
$\beta = t \alpha +(1-t) \gamma$ and $\gamma = \delta$.
\end{enumerate}
\begin{figure}[!ht]
\centering 
\unitlength=0.4mm
\begin{picture}(0,40)(0,-20)
\thicklines
\qbezier(0,-20)(0,-20)(0,20)
\qbezier(-20,0)(-20,0)(-4,0)
\qbezier(20,0)(20,0)(4,0)
\put(0,-10){\vector(0,-1){10}}
\put(-5,-10){\vector(1,0){10}}
\put(-15,5){\vector(0,-1){10}}
\put(25,0){\makebox(0,0)[c,c]{\footnotesize $\beta$}}
\put(5,-20){\makebox(0,0)[c,c]{\footnotesize $\delta$}}
\put(5,20){\makebox(0,0)[c,c]{\footnotesize $\gamma$}}
\put(-25,0){\makebox(0,0)[c,c]{\footnotesize $\alpha$}}
\end{picture}
\caption{Semi-arcs incident to a crossing and the normal orientation.} \label{1ap26}
\end{figure}
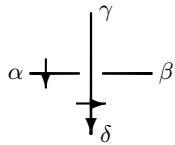

\begin{lemma} \label{crllem}
Let $D$ be a diagram of a welded knot $K$ and $c$ be a classical crossing of $D$. Let $F_{q}=\Lambda _{p}/f(t)$, where $p$ is an odd prime and  $f(t)$ is an irreducible polynomial of degree $d\geq 1$ in $\Lambda_{p}$. If $D'$ is a diagram obtained from $D$ by performing a twist move at crossing $c$, then 
$$
\abs{\dim \operatorname{Col}_{F_{q}}(D') - \dim \operatorname{Col}_{F_{q}}(D)} \leq 1.
$$
\end{lemma}

\begin{proof}
Consider the presentation $\operatorname{Col}_{F_{q}}(D) = \langle SA(D) \, |  \, SR(D)\rangle_{F_{q}}$. Let $D'$ be the diagram obtained from $D$ by performing a twist move at crossing $c$. It is easy to see that the welded crossings that appear by  performing a twist move does not change lebels of semi-arcs,  see Fig.~\ref{2ap26}. 
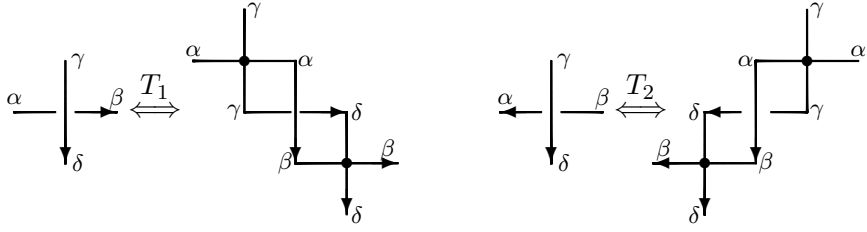
\begin{figure}[!ht]
\centering 
\unitlength=0.34mm
\begin{picture}(0,80)(0,-40)
\thicklines
\qbezier(-150,-20)(-150,-20)(-150,20)
\qbezier(-170,0)(-170,0)(-154,0)
\qbezier(-130,0)(-130,0)(-146,0)
\put(-150,-10){\vector(0,-1){10}}
\put(-140,0){\vector(1,0){10}}
\put(-130,5){\makebox(0,0)[c,c]{\footnotesize $\beta$}}
\put(-145,-20){\makebox(0,0)[c,c]{\footnotesize $\delta$}}
\put(-145,20){\makebox(0,0)[c,c]{\footnotesize $\gamma$}}
\put(-170,5){\makebox(0,0)[c,c]{\footnotesize $\alpha$}}
\put(-115,0){\makebox(0,0)[c,c]{$\Longleftrightarrow$}}
\put(-115,10){\makebox(0,0)[c,c]{$T_{1}$}}
\qbezier(-100,20)(-100,20)(-60,20)
\qbezier(-80,40)(-80,40)(-80,0)
\put(-80,20){\circle*{4}}
\qbezier(-60,20)(-60,20)(-60,-20)
\qbezier(-80,0)(-80,0)(-62,0)
\qbezier(-40,0)(-40,0)(-58,0)
\qbezier(-60,-20)(-60,-20)(-20,-20)
\qbezier(-40,0)(-40,0)(-40,-40)
\put(-40,-20){\circle*{4}}
\put(-60,-10){\vector(0,-1){10}}
\put(-50,0){\vector(1,0){10}}
\put(-76,40){\makebox(0,0)[c,c]{\footnotesize $\gamma$}}
\put(-100,24){\makebox(0,0)[c,c]{\footnotesize $\alpha$}}
\put(-56,20){\makebox(0,0)[c,c]{\footnotesize $\alpha$}}
\put(-84,0){\makebox(0,0)[c,c]{\footnotesize $\gamma$}}
\put(-36,0){\makebox(0,0)[c,c]{\footnotesize $\delta$}}
\put(-64,-20){\makebox(0,0)[c,c]{\footnotesize $\beta$}}
\put(-24,-15){\makebox(0,0)[c,c]{\footnotesize $\beta$}}
\put(-36,-40){\makebox(0,0)[c,c]{\footnotesize $\delta$}}
\put(-30,-20){\vector(1,0){10}}
\put(-40,-30){\vector(0,-1){10}}
\qbezier(40,-20)(40,-20)(40,20)
\qbezier(20,0)(20,0)(36,0)
\qbezier(60,0)(60,0)(44,0)
\put(40,-10){\vector(0,-1){10}}
\put(30,0){\vector(-1,0){10}}
\put(60,5){\makebox(0,0)[c,c]{\footnotesize $\beta$}}
\put(45,-20){\makebox(0,0)[c,c]{\footnotesize $\delta$}}
\put(45,20){\makebox(0,0)[c,c]{\footnotesize $\gamma$}}
\put(22,5){\makebox(0,0)[c,c]{\footnotesize $\alpha$}}
\put(75,0){\makebox(0,0)[c,c]{$\Longleftrightarrow$}}
\put(75,10){\makebox(0,0)[c,c]{$T_{2}$}}
\qbezier(80,-20)(80,-20)(120,-20)
\qbezier(100,-40)(100,-40)(100,0)
\put(100,-20){\circle*{4}}
\qbezier(120,20)(120,20)(120,-20)
\qbezier(100,0)(100,0)(114,0)
\qbezier(140,0)(140,0)(126,0)
\qbezier(120,20)(120,20)(160,20)
\qbezier(140,0)(140,0)(140,40)
\put(140,20){\circle*{4}}
\put(120,-10){\vector(0,-1){10}}
\put(110,0){\vector(-1,0){10}}
\put(144,40){\makebox(0,0)[c,c]{\footnotesize $\gamma$}}
\put(116,20){\makebox(0,0)[c,c]{\footnotesize $\alpha$}}
\put(160,24){\makebox(0,0)[c,c]{\footnotesize $\alpha$}}
\put(144,0){\makebox(0,0)[c,c]{\footnotesize $\gamma$}}
\put(96,0){\makebox(0,0)[c,c]{\footnotesize $\delta$}}
\put(124,-20){\makebox(0,0)[c,c]{\footnotesize $\beta$}}
\put(84,-15){\makebox(0,0)[c,c]{\footnotesize $\beta$}}
\put(96,-40){\makebox(0,0)[c,c]{\footnotesize $\delta$}}
\put(90,-20){\vector(-1,0){10}}
\put(100,-30){\vector(0,-1){10}}
\end{picture}
\caption{Quandle colorings in respect to twist moves.} \label{2ap26}
\end{figure}
Therefore $SA(D) = SA(D')$. We note that for each classical crossing of $D'$ except $c$ the relation remains the same as in $D$. Let us consider the crossing $c \in D$. Let $\alpha, \beta, \gamma$ and $\delta$ be respectively the two underlying semi-arcs and two overlying semi-arcs, such that the normal orientation of $\gamma$ points from $\alpha$ to $\beta$ and that of $\alpha$ points from $\gamma$ to $\delta$ as defined in Fig.~\ref{1ap26}. Therefore, relations in $SR(D)$ corresponding to $c \in D$ are
$$
\gamma = \delta \qquad \mbox{and}  \qquad \beta = t \alpha + (1-t) \gamma.
$$
Hence we can replace the relation $\beta = t \alpha +(1-t) \gamma$ with $\beta = t\alpha + \delta -t \gamma$ or $\beta = t\alpha +\gamma - t \delta$.  Therefore,  $\operatorname{Col}_{F_{q}} (D) \cong \langle SA(D) \, | \, \widehat{SR} (D) \rangle_{F_{q}}$ and $\operatorname{Col}_{F_{q}} (D) \cong \langle SA(D) \, | \, \overline{SR }(D)  \rangle_{F_{q}}$, where
$$
\widehat{SR} (D) = SR(D)\cup \{\beta =t \alpha + \delta -t \gamma \} \setminus \{\beta= t \alpha +(1-t) \gamma \}
$$
and
$$
\overline{SR} (D) = SR (D)\cup \{ \beta = t \alpha + \gamma - t \delta \} \setminus \{ \beta = t \alpha + (1-t) \gamma \}.
$$
Consider twist moves $T_{1}$ and $T_{2}$ according to Fig.~\ref{2ap26}.  

\underline{(1) The case of $T_1$-move.} Relations in $SR (D')$, corresponding to $c \in D'$, are $\alpha = \beta$ and $\delta = t \gamma + (1-t)\alpha$. Hence we can replace the relation $\delta = t \gamma +(1-t) \alpha$ with $\delta  = t \gamma + \beta- t\alpha$. Therefore, $\operatorname{Col}_{F_{q}}(D')\cong \langle SA (D') \, | \, \widehat{SR} (D')\rangle_{F_{q}}$, where
$$
\widehat{SR} (D')= SR (D')\cup \{\delta = t\gamma +\beta-t\alpha\}\setminus \{\delta=t\gamma+(1-t)\alpha\}.\
$$  
Observe that $\delta = t \gamma +\beta-t\alpha$ is equivalent to  $\beta = t\alpha +\delta-t\gamma$. Therefore 
$$
\widehat{SR} (D') = \widehat{SR} (D) \cup\{\alpha = \beta \}\setminus \{\gamma = \delta \}
$$ 
and $\dim \operatorname{Col}_{F_{q}}(D')= \dim \operatorname{Col}_{F_{q}}(D) + r$, where $r \in \{-1,0,1\}$. 

\underline{(2) The case of $T_2$-move.} Relators in $SR (D')$, corresponding to $c \in D'$, are
$\alpha = \beta$  and $\gamma = t \delta +(1-t) \beta$.  
Hence we can replace the relation $\gamma = t \delta +(1-t) \beta$ with $\gamma =t \delta +\beta -t \alpha$ such that $\operatorname{Col}_{F_{q}}(D') \cong \langle SA (D') \, | \, \overline{SR} (D') \rangle_{F_{q}}$, where
$$
\overline{SR} (D')= SR (D') \cup \{\gamma=t \delta+\beta-t\alpha\} \setminus \{\gamma =t \delta +(1-t)\beta \}.
$$ 
Observe that $\gamma = t \delta + \beta - t \alpha$ is equivalent to $\beta = t \alpha + \gamma - t \delta$. Therefore  
$$
\overline{SR} (D') = \overline{SR} (D') \cup \{\alpha = \beta \} \setminus \{\gamma = \delta \}
$$ 
and $\dim \operatorname{Col}_{F_{q}}(D) = \dim \operatorname{Col}_{F_{q}}(D')+r$, where  $r \in \{-1,0,1\}$. 

Hence, for both twist moves $T_{1}$ and $T_{2}$ we obtain 
$$
\abs{\dim \operatorname{Col}_{F_{q}}(D') - \dim \operatorname{Col}_{F_{q}}(D)}  \leq 1
$$
and lemma is proved.
\end{proof}

Analogously \cite{unknottingweldedLi2017} we come up at a lower bound for twist number in terms of the dimension of the coloring space by quandle $F_q$.
 
\begin{theorem} \label{crlthm}
Let $D$ be a diagram of a welded knot $K$ and $\operatorname{Col}_{F_{q}} (D)$ be the set of colorings of $D$ by $F_{q}$. Then the following inequality holds: 
$$
\dim \operatorname{Col}_{F_{q}}(D)-1\leq ut(K).
$$
\end{theorem}

\begin{proof} 
Let $D$ be a diagram of welded knot $K$ which realize $ut(K)$. Let $ut(K)=n$ and $D=D_{0}, D_{1},\ldots, D_{n}$ be a sequence of diagrams such that each $D_{i}$ is obtained from $D_{i+1}$ by applying one twist move and $D_{n}$ is a diagram of the trivial welded knot. Therefore, using Lemma~\ref{crllem}, we have
$$
\begin{gathered} 
\left| \dim \operatorname{Col}_{F_{q}}(D_0) - \dim \operatorname{Col}_{F_{q}}(D_n) \right| \\  
=  \Big| \displaystyle  \sum _{ i=0}^{n-1} \big(\dim  \operatorname{Col}_{F_{q}}(D_i) - \dim \operatorname{Col}_{F_{q}}(D_{i+1})\big) \Big| 
\\  \leq  \displaystyle  \sum _{ i=0}^{n-1} \left| \dim \operatorname{Col}_{F_{q}}(D_i) - \dim \operatorname{Col}_{F_{q}}(D_{i+1}) \right| \leq n.                        
\end{gathered}
$$
Since $\dim \operatorname{Col}_{F_{q}}(D_n)=1$ and $\dim  \operatorname{Col}_{F_{q}}(D_0) \geq \dim \operatorname{Col}_{F_{q}}(D_n)$, we get $\dim  \operatorname{Col}_{F_{q}}(D)-1\leq ut(K)$. 
 \end{proof} 
 
Denote by $B_{n}$ welded knot that admits diagram $D_{n}$ with $3n$ classical crossings and $2n$ welded crossings as presented in Fig.~\ref{4ap26}.  Recall that $B_{1}$ was presented in Fig.~\ref{utk} and $B_{n}$ can be considered as a connected sum of $n$ copies of $B_{1}$. 
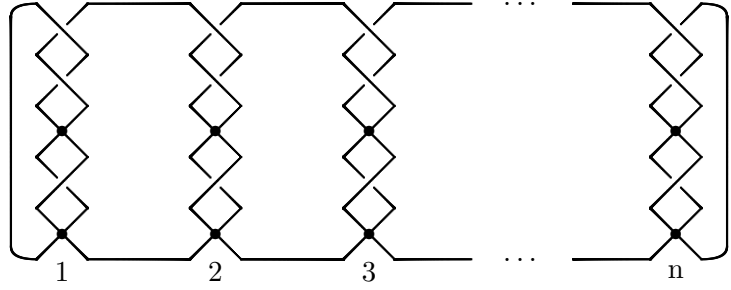
\begin{figure}[!ht]
\centering 
\unitlength=0.34mm
\begin{picture}(0,100)(0,-50)
\thicklines
\qbezier(-130,-50)(-130,-50)(-110,-30)
\qbezier(-110,-50)(-110,-50)(-130,-30) 
\put(-120,-40){\circle*{4}}
\qbezier(-130,-30)(-130,-30)(-110,-10)
\qbezier(-130,-10)(-130,-10)(-122,-18) 
\qbezier(-110,-30)(-110,-30)(-118,-22)
\qbezier(-130,10)(-130,10)(-110,-10)
\qbezier(-130,-10)(-130,-10)(-110,10) 
\put(-120,0){\circle*{4}}
\qbezier(-130,30)(-130,30)(-110,10)
\qbezier(-130,10)(-130,10)(-122,18) 
\qbezier(-110,30)(-110,30)(-118,22)
\qbezier(-130,50)(-130,50)(-110,30)
\qbezier(-130,30)(-130,30)(-122,38) 
\qbezier(-110,50)(-110,50)(-118,42)
\qbezier(-130,50)(-140,50)(-140,45)
\qbezier(-140,45)(-140,45)(-140,-45)
\qbezier(-140,-45)(-140,-50)(-130,-50)
\qbezier(-110,50)(-70,50)(-70,50)
\qbezier(-110,-50)(-70,-50)(-70,-50)
\put(-120,-55){\makebox(0,0)[c,c]{1}}
\qbezier(-70,-50)(-70,-50)(-50,-30)
\qbezier(-70,-30)(-70,-30)(-50,-50) 
\put(-60,-40){\circle*{4}}
\qbezier(-70,-30)(-70,-30)(-50,-10)
\qbezier(-70,-10)(-70,-10)(-62,-18) 
\qbezier(-50,-30)(-50,-30)(-58,-22)
\qbezier(-70,10)(-70,10)(-50,-10)
\qbezier(-70,-10)(-70,-10)(-50,10) 
\put(-60,0){\circle*{4}}
\qbezier(-70,30)(-70,30)(-50,10)
\qbezier(-70,10)(-70,10)(-62,18) 
\qbezier(-50,30)(-50,30)(-58,22)
\qbezier(-70,50)(-70,50)(-50,30)
\qbezier(-70,30)(-70,30)(-62,38) 
\qbezier(-50,50)(-50,50)(-58,42)
\qbezier(-50,50)(-10,50)(-10,50)
\qbezier(-50,-50)(-10,-50)(-10,-50)
\put(-60,-55){\makebox(0,0)[c,c]{2}}
\qbezier(-10,-50)(-10,-50)(10,-30)
\qbezier(-10,-30)(-10,-30)(10,-50) 
\put(0,-40){\circle*{4}}
\qbezier(-10,-30)(-10,-30)(10,-10)
\qbezier(-10,-10)(-10,-10)(-2,-18) 
\qbezier(10,-30)(10,-30)(2,-22)
\qbezier(-10,10)(-10,10)(10,-10)
\qbezier(-10,-10)(-10,-10)(10,10) 
\put(0,0){\circle*{4}}
\qbezier(-10,30)(-10,30)(10,10)
\qbezier(-10,10)(-10,10)(-2,18) 
\qbezier(10,30)(10,30)(2,22)
\qbezier(-10,50)(-10,50)(10,30)
\qbezier(-10,30)(-10,30)(-2,38) 
\qbezier(10,50)(10,50)(2,42)
\qbezier(-50,50)(-10,50)(-10,50)
\qbezier(-50,-50)(-10,-50)(-10,-50)
\put(0,-55){\makebox(0,0)[c,c]{3}}
\qbezier(10,50)(10,50)(40,50)
\qbezier(10,-50)(10,-50)(40,-50)
\put(60,50){\makebox(0,0)[c,c]{$\dots$}}
\put(60,-50){\makebox(0,0)[c,c]{$\dots$}}
\qbezier(80,50)(80,50)(110,50)
\qbezier(80,-50)(80,-50)(110,-50)
\qbezier(110,-50)(110,-50)(130,-30)
\qbezier(110,-30)(110,-30)(130,-50) 
\put(120,-40){\circle*{4}}
\qbezier(110,-30)(110,-30)(130,-10)
\qbezier(110,-10)(110,-10)(118,-18) 
\qbezier(130,-30)(130,-30)(122,-22)
\qbezier(110,10)(110,10)(130,-10)
\qbezier(110,-10)(110,-10)(130,10) 
\put(120,0){\circle*{4}}
\qbezier(110,30)(110,30)(130,10)
\qbezier(110,10)(110,10)(118,18) 
\qbezier(130,30)(130,30)(122,22)
\qbezier(110,50)(110,50)(130,30)
\qbezier(110,30)(110,30)(118,38) 
\qbezier(130,50)(130,50)(122,42)
\qbezier(130,50)(140,50)(140,45)
\qbezier(130,-50)(140,-50)(140,-45)
\qbezier(140,45)(140,45)(140,-45)
\put(120,-55){\makebox(0,0)[c,c]{n}}
\end{picture}
\caption{Diagram $D_{n}$ of a welded knot $B_{n}$.} \label{4ap26}
\end{figure}

\begin{lemma} \label{lemma3}
The following equality holds: $ut (B_{n})= n$. 
\end{lemma}

\begin{proof}
It was proved in \cite[pf.~9.2]{unknottingweldedLi2017} that $\dim \operatorname{Col}_{R_{3}}(D_{n})=n+1$. Therefore by Theorem~\ref{crlthm}, we have $n\leq ut(B_{n})$. Let us perform twist moves at $n$ crossings lying between welded crossings  in diagram $D_{n}$ and then Reidemeister moves VRII and RII.  
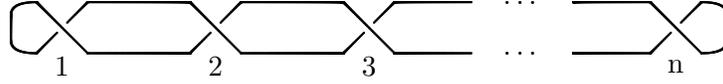
\begin{figure}[!ht]
\centering 
\unitlength=0.34mm
\begin{picture}(0,30)(0,30)
\thicklines
\qbezier(-130,50)(-130,50)(-110,30)
\qbezier(-130,30)(-130,30)(-122,38) 
\qbezier(-110,50)(-110,50)(-118,42)
\qbezier(-130,50)(-140,50)(-140,45)
\qbezier(-140,45)(-140,45)(-140,35)
\qbezier(-140,35)(-140,30)(-130,30)
\qbezier(-110,50)(-70,50)(-70,50)
\qbezier(-110,30)(-70,30)(-70,30)
\put(-120,25){\makebox(0,0)[c,c]{1}}
\qbezier(-70,50)(-70,50)(-50,30)
\qbezier(-70,30)(-70,30)(-62,38) 
\qbezier(-50,50)(-50,50)(-58,42)
\qbezier(-50,50)(-10,50)(-10,50)
\qbezier(-50,30)(-10,30)(-10,30)
\put(-60,25){\makebox(0,0)[c,c]{2}}
\qbezier(-10,50)(-10,50)(10,30)
\qbezier(-10,30)(-10,30)(-2,38) 
\qbezier(10,50)(10,50)(2,42)
\put(0,25){\makebox(0,0)[c,c]{3}}
\qbezier(10,50)(10,50)(40,50)
\qbezier(10,30)(10,30)(40,30)
\put(60,50){\makebox(0,0)[c,c]{$\dots$}}
\put(60,30){\makebox(0,0)[c,c]{$\dots$}}
\qbezier(80,50)(80,50)(110,50)
\qbezier(80,30)(80,30)(110,30)
\qbezier(110,50)(110,50)(130,30)
\qbezier(110,30)(110,30)(118,38) 
\qbezier(130,50)(130,50)(122,42)
\qbezier(130,50)(140,50)(140,45)
\qbezier(130,30)(140,30)(140,35)
\qbezier(140,45)(140,45)(140,35)
\put(120,25){\makebox(0,0)[c,c]{n}}
\end{picture}
\caption{Diagram of a trivial welded knot.} \label{5ap26}
\end{figure}
 Then the resulting diagram is equivalent to the diagram as shown in Fig.~\ref{5ap26}, which is a trivial welded knot diagram. Therefore $ut(B_{n})\leq n$. 
\end{proof}

\section{Gordian complex of welded knots}

It is known that unknotting moves for knots and virtual knots admit to complexes of knots or links, see, for example~\cite{GPV}. 
Z.~Li, F.~Lei and J.~Wu \cite{unknottingweldedLi2017} defined the distance between welded knots using crossing change, $\sharp$-move and $\Delta$-move. They provided explicit examples of welded knots $\{K,K'\}$ such that distance by crossing change is $d_{X}(K,K')=1$. Also it was shown that for all $m,n \in \mathbb{N}$, welded knots $K_m, K_n$ exist such that $d_{X}(K_m,K_n)\leq \mid m-n \mid$.

As we demonstrated above, the twist move along with welded Reidemeister moves is also an unknotting operation for welded knots. Therefore for any pair of welded knots $\{K,K'\}$, a diagram $D$ of $K$ can be converted into a diagram $D'$ of $K'$ by a finite sequence of twist moves and welded Reidemeister moves. We define a \emph{twist-distance} $d_T(K,K')$ between two welded knots $K$ and $K'$ as the  minimum number of twist moves required to convert $D$ into $D'$, where minimum is taken over all diagrams $D$ of $K$ and $D'$ of $K'$. \emph{Gordian complex} $\mathcal{G}_T$ of welded knots by twist move is defined by considering the set of all welded knot isotopy classes as \emph{vertex set} of $\mathcal{G}_T$ and a set of welded knots $\{K_0,...,K_n\}$ spans an $n$-simplex if and only if $d_T(K_i,K_j)=1$ for all $i\neq j \in \{0,1,\ldots,n\}$.

\begin{remark}
{\rm 
For any welded knot $K$, we have $d_T(K,U) = u_t(K)$, i.e, the twist number of $K$ where $U$ is unknot. Thus if $ut(K)=1$, the pair $\{K,U\}$ spans an 1-simplex in $\mathcal{G}_T$. For the welded knot $W_{n} = {\bf b} (2n + \frac{1}{2})$ given in Proposition~\ref{prop4.1}, $\{W_{n},U\}$ spans an 1-simplex in $\mathcal{G}_T$.
}
\end{remark} 

There are infinitely many pairs $\{K,K'\}$  such that $d_T(K,K')=1$, where $K$ is a non-trivial welded knot and $K'$ is a non-trivial classical knot.  

\begin{proposition}
$\mathcal{G}_T$ contains infinitely many 1-simplexes with one vertex to be a non-classical welded knot and other vertex to be a classical knot. 
\end{proposition}

\begin{proof}
Consider the family of welded knots $A_{2n}$ shown in Fig.~\ref{K2N} with classical crossings numerated by $1, 2, \ldots, 2n$. 
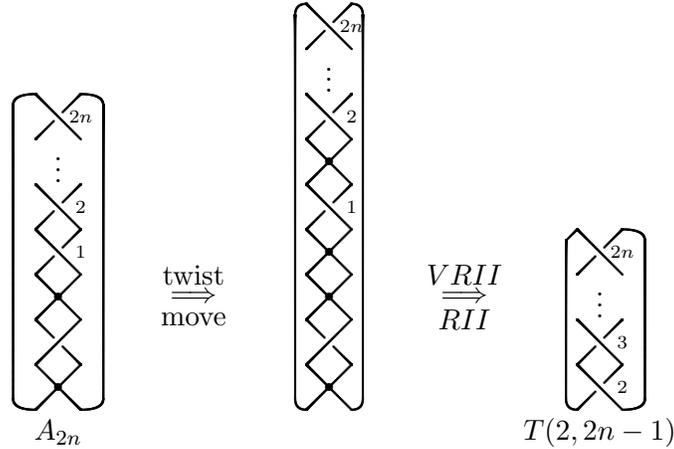
\begin{figure}[!ht]
\centering 
\unitlength=0.3mm
\begin{picture}(0,190)(0,-60)
\thicklines
\qbezier(-130,-50)(-130,-50)(-110,-30)
\qbezier(-110,-50)(-110,-50)(-130,-30) 
\put(-120,-40){\circle*{4}}
\qbezier(-130,-30)(-130,-30)(-110,-10)
\qbezier(-130,-10)(-130,-10)(-122,-18) 
\qbezier(-110,-30)(-110,-30)(-118,-22)
\qbezier(-130,10)(-130,10)(-110,-10)
\qbezier(-130,-10)(-130,-10)(-110,10) 
\put(-120,0){\circle*{4}}
\qbezier(-130,30)(-130,30)(-110,10)
\qbezier(-130,10)(-130,10)(-122,18) 
\qbezier(-110,30)(-110,30)(-118,22)
\qbezier(-130,50)(-130,50)(-110,30)
\qbezier(-130,30)(-130,30)(-122,38) 
\qbezier(-110,50)(-110,50)(-118,42)
\put(-120,60){\makebox(0,0)[c,c]{$\vdots$}}
\qbezier(-130,90)(-130,90)(-110,70)
\qbezier(-130,70)(-130,70)(-122,78) 
\qbezier(-110,90)(-110,90)(-118,82)
\qbezier(-130,90)(-140,90)(-140,85)
\qbezier(-140,85)(-140,85)(-140,-45)
\qbezier(-140,-45)(-140,-50)(-130,-50)
\qbezier(-110,90)(-100,90)(-100,85)
\qbezier(-100,85)(-100,85)(-100,-45) 
\qbezier(-100,-45)(-100,-50)(-110,-50)
\put(-120,-60){\makebox(0,0)[c,c]{$A_{2n}$}}
\put(-110,80){\makebox(0,0)[c,c]{\tiny $2n$}}
\put(-110,40){\makebox(0,0)[c,c]{\tiny $2$}}
\put(-110,20){\makebox(0,0)[c,c]{\tiny $1$}}
\put(-60,0){\makebox(0,0)[c,c]{$\Longrightarrow$}}
\put(-60,10){\makebox(0,0)[c,c]{twist}}
\put(-60,-10){\makebox(0,0)[c,c]{move}}
\qbezier(-10,-50)(-10,-50)(10,-30)
\qbezier(-10,-30)(-10,-30)(10,-50) 
\put(0,-40){\circle*{4}}
\qbezier(-10,-30)(-10,-30)(10,-10)
\qbezier(-10,-10)(-10,-10)(-2,-18) 
\qbezier(10,-30)(10,-30)(2,-22)
\qbezier(-10,10)(-10,10)(10,-10)
\qbezier(-10,-10)(-10,-10)(10,10) 
\put(0,0){\circle*{4}}
\qbezier(-10,30)(-10,30)(10,10)
\qbezier(-10,10)(-10,10)(10,30) 
\put(0,20){\circle*{4}}
\qbezier(-10,50)(-10,50)(10,30)
\qbezier(-10,30)(-10,30)(-2,38) 
\qbezier(10,50)(10,50)(2,42)
\qbezier(-10,70)(-10,70)(10,50)
\qbezier(-10,50)(-10,50)(10,70) 
\put(0,60){\circle*{4}}
\qbezier(-10,90)(-10,90)(10,70)
\qbezier(-10,70)(-10,70)(-2,78) 
\qbezier(10,90)(10,90)(2,82)
\put(0,100){\makebox(0,0)[c,c]{$\vdots$}}
\qbezier(-10,130)(-10,130)(10,110)
\qbezier(-10,110)(-10,110)(-2,118) 
\qbezier(10,130)(10,130)(2,122)
\put(10,120){\makebox(0,0)[c,c]{\tiny $2n$}}
\put(10,80){\makebox(0,0)[c,c]{\tiny $2$}}
\put(10,40){\makebox(0,0)[c,c]{\tiny $1$}}
\put(60,0){\makebox(0,0)[c,c]{$\Longrightarrow$}}
\put(60,10){\makebox(0,0)[c,c]{$VRII$}}
\put(60,-10){\makebox(0,0)[c,c]{$RII$}}
\qbezier(-10,130)(-15,130)(-15,125)
\qbezier(-15,125)(-15,125)(-15,-45)
\qbezier(-15,-45)(-15,-50)(-10,-50)
\qbezier(10,130)(15,130)(15,125)
\qbezier(15,125)(15,125)(15,-45)
\qbezier(15,-45)(15,-50)(10,-50)
\qbezier(110,-30)(110,-30)(130,-50)
\qbezier(110,-50)(110,-50)(118,-42) 
\qbezier(130,-30)(130,-30)(122,-38)
\qbezier(110,-10)(110,-10)(130,-30)
\qbezier(110,-30)(110,-30)(118,-22) 
\qbezier(130,-10)(130,-10)(122,-18)
\put(120,0){\makebox(0,0)[c,c]{$\vdots$}}
\qbezier(110,30)(110,30)(130,10)
\qbezier(110,10)(110,10)(118,18) 
\qbezier(130,30)(130,30)(122,22)
\put(130,20){\makebox(0,0)[c,c]{\tiny $2n$}}
\put(130,-20){\makebox(0,0)[c,c]{\tiny $3$}}
\put(130,-40){\makebox(0,0)[c,c]{\tiny $2$}}
\qbezier(130,30)(140,30)(140,25)
\qbezier(130,-50)(140,-50)(140,-45)
\qbezier(140,25)(140,25)(140,-45)
\qbezier(110,30)(110,30)(105,25)
\qbezier(110,-50)(105,-50)(105,-45)
\qbezier(105,25)(105,25)(105,-45)
\put(120,-60){\makebox(0,0)[c,c]{$T(2,2n-1)$}}
\end{picture}
\caption{From welded knot $A_{2n}$ to torus knot $T_{2n-1}$.} \label{K2N}
\end{figure}
By computing the commutators of the fundamental groups of $A_{2n}$ it was proved in~\cite{n} that all $A_{2n}$ are distinct and non-classical welded knots. Consider the welded knot obtained by applying one twist move at crossing number ``$1$' 'in $A_{2n}$. The obtained welded knot can be simplified by moves VRII and RII. Then we will get a classical knot with crossings numerated by $2, 3, \ldots, 2n$ which is the torus knot $T(2, 2n-1)$. Welded knot $A_{2n}$ is not classical, but $T(2, 2n-1)$ is classical knot. Therefore they are distinct. And $T_{2n-1}$ is obtained from $A_{2n}$ by using a twist move,  hence $d_T(A_{2n},T_{2n-1})=1$. Therefore for each $n \in \mathbb{N}$ the pair $\{A_{2n},T_{2n-1}\}$ spans 1-simplex in $\mathcal{G}_T$.  
\end{proof} 

We say two welded knots $\{K,K'\}$ are connected by a path of length $n$ in $\mathcal{G}_T$ if there exists welded knots $\{K_0 ,K_1, \ldots, K_n\}$ such that $d_T(K_i,K_{i+1})=1$ for each $i=0,1,\ldots, n-1$ and $K_0=K, K_n= K'$. The minimum length among all the paths between $K$ and $K'$ gives the distance $d_T(K,K')$. Since twist move is an unknotting operation, there exist a path between any two distinct welded knots and hence $\mathcal{G}_T$ is connected complex. We show the existence of an infinite length path in $\mathcal{G}_T$ by providing an infinite family of welded knots $\{K_1, K_2, \ldots \}$ such that $d_{T}(K_i,K_{i+1}) = 1$ for all $i \geq 1$.

\begin{proposition}
There exists a family $\{B_1,  B_2, \ldots \}$ of distinct welded knots in $\mathcal{G}_T$ such that for all $i=1,2,\ldots$ the distance $d_{T} (B_i, B_{i+1}) = 1$.
\end{proposition}

\begin{proof}
Let $\{B_n\}_{n\geq 1}$ be the family of welded knots given by Fig.~\ref{INFPATH}. As shown in Lemma~\ref{lemma3}, for each $n=1,2,\ldots$ twist number is known: $ut(B_n)= n$ thus, all the $\{B_n\}$ are distinct welded knots. 
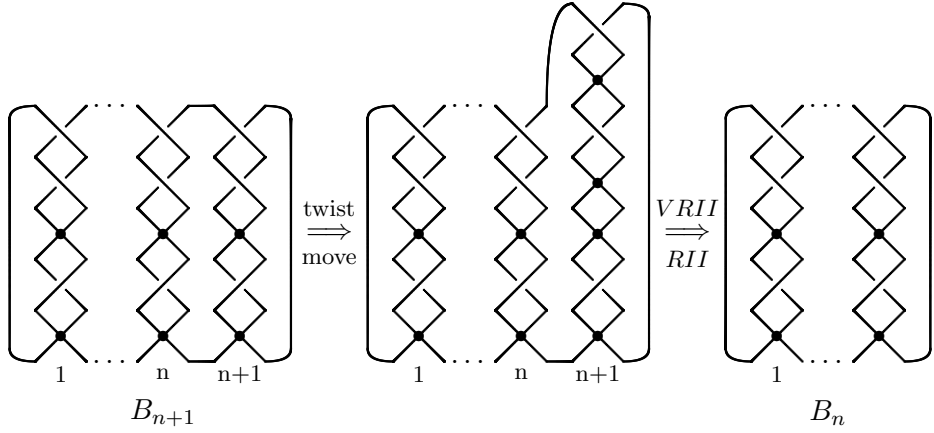
\begin{figure}[!ht]
\centering 
\unitlength=0.34mm
\begin{picture}(360,160)(-190,-70)
\thicklines
\qbezier(-180,-50)(-180,-50)(-160,-30)
\qbezier(-160,-50)(-160,-50)(-180,-30) 
\put(-170,-40){\circle*{4}}
\qbezier(-180,-30)(-180,-30)(-160,-10)
\qbezier(-180,-10)(-180,-10)(-172,-18) 
\qbezier(-160,-30)(-160,-30)(-168,-22)
\qbezier(-180,10)(-180,10)(-160,-10)
\qbezier(-180,-10)(-180,-10)(-160,10) 
\put(-170,0){\circle*{4}}
\qbezier(-180,30)(-180,30)(-160,10)
\qbezier(-180,10)(-180,10)(-172,18) 
\qbezier(-160,30)(-160,30)(-168,22)
\qbezier(-180,50)(-180,50)(-160,30)
\qbezier(-180,30)(-180,30)(-172,38) 
\qbezier(-160,50)(-160,50)(-168,42)
\qbezier(-180,50)(-190,50)(-190,45)
\qbezier(-190,45)(-190,45)(-190,-45)
\qbezier(-190,-45)(-190,-50)(-180,-50)
\put(-150,50){\makebox(0,0)[c,c]{$\dots$}}
\put(-150,-50){\makebox(0,0)[c,c]{$\dots$}}
\put(-170,-55){\makebox(0,0)[c,c]{\footnotesize 1}}
\qbezier(-140,-50)(-140,-50)(-120,-30)
\qbezier(-140,-30)(-140,-30)(-120,-50) 
\put(-130,-40){\circle*{4}}
\qbezier(-140,-30)(-140,-30)(-120,-10)
\qbezier(-140,-10)(-140,-10)(-132,-18) 
\qbezier(-120,-30)(-120,-30)(-128,-22)
\qbezier(-140,10)(-140,10)(-120,-10)
\qbezier(-140,-10)(-140,-10)(-120,10) 
\put(-130,0){\circle*{4}}
\qbezier(-140,30)(-140,30)(-120,10)
\qbezier(-140,10)(-140,10)(-132,18) 
\qbezier(-120,30)(-120,30)(-128,22)
\qbezier(-140,50)(-140,50)(-120,30)
\qbezier(-140,30)(-140,30)(-132,38) 
\qbezier(-120,50)(-120,50)(-128,42)
\qbezier(-120,50)(-110,50)(-110,50)
\qbezier(-120,-50)(-110,-50)(-110,-50)
\put(-130,-55){\makebox(0,0)[c,c]{\footnotesize n}}
\qbezier(-110,-50)(-110,-50)(-90,-30)
\qbezier(-110,-30)(-110,-30)(-90,-50) 
\put(-100,-40){\circle*{4}}
\qbezier(-110,-30)(-110,-30)(-90,-10)
\qbezier(-110,-10)(-110,-10)(-102,-18) 
\qbezier(-90,-30)(-90,-30)(-98,-22)
\qbezier(-110,10)(-110,10)(-90,-10)
\qbezier(-110,-10)(-110,-10)(-90,10) 
\put(-100,0){\circle*{4}}
\qbezier(-110,30)(-110,30)(-90,10)
\qbezier(-110,10)(-110,10)(-102,18) 
\qbezier(-90,30)(-90,30)(-98,22)
\qbezier(-110,50)(-110,50)(-90,30)
\qbezier(-110,30)(-110,30)(-102,38) 
\qbezier(-90,50)(-90,50)(-98,42)
\put(-100,-55){\makebox(0,0)[c,c]{\footnotesize n+1}}
\qbezier(-90,50)(-80,50)(-80,45)
\qbezier(-80,45)(-80,45)(-80,-45) 
\qbezier(-80,-45)(-80,-50)(-90,-50)
\qbezier(-40,-50)(-40,-50)(-20,-30)
\qbezier(-20,-50)(-20,-50)(-40,-30) 
\put(-30,-40){\circle*{4}}
\qbezier(-40,-30)(-40,-30)(-20,-10)
\qbezier(-40,-10)(-40,-10)(-32,-18) 
\qbezier(-20,-30)(-20,-30)(-28,-22)
\qbezier(-40,10)(-20,-10)(-20,-10)
\qbezier(-40,-10)(-20,10)(-20,10) 
\put(-30,0){\circle*{4}}
\qbezier(-40,30)(-40,30)(-20,10)
\qbezier(-40,10)(-40,10)(-32,18) 
\qbezier(-20,30)(-20,30)(-28,22)
\qbezier(-40,50)(-40,50)(-20,30)
\qbezier(-40,30)(-40,30)(-32,38) 
\qbezier(-20,50)(-20,50)(-28,42)
\qbezier(-40,50)(-50,50)(-50,45)
\qbezier(-50,45)(-50,45)(-50,-45)
\qbezier(-50,-45)(-50,-50)(-40,-50)
\put(-10,50){\makebox(0,0)[c,c]{$\dots$}}
\put(-10,-50){\makebox(0,0)[c,c]{$\dots$}}
\put(-30,-55){\makebox(0,0)[c,c]{\footnotesize 1}}
\qbezier(0,-50)(0,-50)(20,-30)
\qbezier(0,-30)(0,-30)(20,-50) 
\put(10,-40){\circle*{4}}
\qbezier(0,-30)(0,-30)(20,-10)
\qbezier(0,-10)(0,-10)(8,-18) 
\qbezier(20,-30)(20,-30)(12,-22)
\qbezier(0,10)(0,10)(20,-10)
\qbezier(0,-10)(0,-10)(20,10) 
\put(10,0){\circle*{4}}
\qbezier(0,30)(0,30)(20,10)
\qbezier(0,10)(0,10)(8,18) 
\qbezier(20,30)(20,30)(12,22)
\qbezier(0,50)(0,50)(20,30)
\qbezier(0,30)(0,30)(8,38) 
\qbezier(20,50)(20,50)(12,42)
\qbezier(20,50)(20,90)(30,90)
\qbezier(20,-50)(30,-50)(30,-50)
\put(10,-55){\makebox(0,0)[c,c]{\footnotesize n}}
\qbezier(30,-50)(30,-50)(50,-30)
\qbezier(30,-30)(30,-30)(50,-50) 
\put(40,-40){\circle*{4}}
\qbezier(30,-30)(30,-30)(50,-10)
\qbezier(30,-10)(30,-10)(38,-18) 
\qbezier(50,-30)(50,-30)(42,-22)
\qbezier(30,10)(30,10)(50,-10)
\qbezier(30,-10)(30,-10)(50,10) 
\put(40,0){\circle*{4}}
\qbezier(30,30)(30,30)(50,10)
\qbezier(30,10)(30,10)(50,30) 
\put(40,20){\circle*{4}}
\qbezier(30,50)(30,50)(50,30)
\qbezier(30,30)(30,30)(38,38) 
\qbezier(50,50)(50,50)(42,42)
\qbezier(30,70)(30,70)(50,50)
\qbezier(30,50)(30,50)(50,70) 
\put(40,60){\circle*{4}}
\qbezier(30,90)(30,90)(50,70)
\qbezier(30,70)(30,70)(38,78) 
\qbezier(50,90)(50,90)(42,82)
\put(40,-55){\makebox(0,0)[c,c]{\footnotesize n+1}}
\qbezier(50,90)(60,90)(60,85)
\qbezier(60,85)(60,85)(60,-45) 
\qbezier(60,-45)(60,-50)(50,-50)
\qbezier(100,-50)(90,-50)(90,-45)
\qbezier(90,-45)(90,-45)(90,45) 
\qbezier(90,45)(90,50)(100,50) 
\qbezier(100,-50)(100,-50)(120,-30)
\qbezier(100,-30)(100,-30)(120,-50) 
\put(110,-40){\circle*{4}}
\qbezier(100,-30)(100,-30)(120,-10)
\qbezier(100,-10)(100,-10)(108,-18) 
\qbezier(120,-30)(120,-30)(112,-22)
\qbezier(100,10)(100,10)(120,-10)
\qbezier(100,-10)(100,-10)(120,10) 
\put(110,0){\circle*{4}}
\qbezier(100,30)(100,30)(120,10)
\qbezier(100,10)(100,10)(108,18) 
\qbezier(120,30)(120,30)(112,22)
\qbezier(100,50)(100,50)(120,30)
\qbezier(100,30)(100,30)(108,38) 
\qbezier(120,50)(120,50)(112,42)
\put(110,-55){\makebox(0,0)[c,c]{\footnotesize 1}}
\put(130,50){\makebox(0,0)[c,c]{$\dots$}}
\put(130,-50){\makebox(0,0)[c,c]{$\dots$}}
\qbezier(140,-50)(140,-50)(160,-30)
\qbezier(140,-30)(140,-30)(160,-50) 
\put(150,-40){\circle*{4}}
\qbezier(140,-30)(140,-30)(160,-10)
\qbezier(140,-10)(140,-10)(148,-18) 
\qbezier(160,-30)(160,-30)(152,-22)
\qbezier(140,10)(140,10)(160,-10)
\qbezier(140,-10)(140,-10)(160,10) 
\put(150,0){\circle*{4}}
\qbezier(140,30)(140,30)(160,10)
\qbezier(140,10)(140,10)(148,18) 
\qbezier(160,30)(160,30)(152,22)
\qbezier(140,50)(140,50)(160,30)
\qbezier(140,30)(140,30)(148,38) 
\qbezier(160,50)(160,50)(152,42)
\qbezier(160,50)(170,50)(170,45)
\qbezier(170,45)(170,45)(170,-45) 
\qbezier(170,-45)(170,-50)(160,-50)
\put(-130,-70){\makebox(0,0)[c,c]{$B_{n+1}$}}
\put(130,-70){\makebox(0,0)[c,c]{$B_{n}$}}
\put(-65,0){\makebox(0,0)[c,c]{$\Longrightarrow$}}
\put(-65,10){\makebox(0,0)[c,c]{\footnotesize twist}}
\put(-65,-10){\makebox(0,0)[c,c]{\footnotesize move}}
\put(75,0){\makebox(0,0)[c,c]{$\Longrightarrow$}}
\put(75,10){\makebox(0,0)[c,c]{\footnotesize $VRII$}}
\put(75,-10){\makebox(0,0)[c,c]{\footnotesize $RII$}}
\end{picture}
\caption{The twist move from $B_{n+1}$ to $B_{n}$.} \label{INFPATH}
\end{figure}
Observe from the Fig.~\ref{INFPATH} that each $B_n$ can be obtained from $B_{n+1}$ by using one twist move and  welded Reidemeister moves VRII and RII. Therefore, $d_T(B_n,  B_{n+1})=1$ for $n=1,2,\ldots$ as required.
\end{proof}

For the family $\{B_n\}_{n\geq 1}$ given in Fig.~\ref{INFPATH}, if $m\neq n$ then it is possible to convert $B_m$ into $B_n$ using $\mid m-n \mid$ twist moves, i.e , $d_T (B_m, B_n)\leq \mid m-n\mid$ with equality being still unknown. 

For a Gordian complex of knots it is interesting to ask the question that whether it contains an $n$-simplex of each dimension $n \in \mathbb{N}$ or not, there are examples of Gordian complexes with answer to this question in yes as well as no for different local moves. If it is possible to construct an infinite family of knots $\{K_n\}_{n\geq 0}$ such that distance by corresponding local move $d(K_m,K_n)=1$ for distinct $m$ and $n$ then $\sigma_n = \{K_0, K_1, \ldots, K_n \}$ forms an $n$-simplex.
We propose a family $\{WK_n\}_{n\geq 0}$ of welded knots which might be such an example in the case of twist move, however it is not yet known to us whether $WK_n$ are all distinct welded knots.

\begin{theorem} \label{thm6.1} 
The Gordian complex $\mathcal{G}_T$ contains an infinite family of welded knots $\{WK_n\}_{n\geq 0}$ satisfying $d_T(WK_m,WK_n)\leq 1$ for distinct integer $m,n \geq 0$.
\end{theorem}

\begin{proof}
Denote by $WK_0$ the trivial welded knot and by $WK_n$, $n \geq 1, $ the welded knot given in Fig.~\ref{WKN}. Observe that $WK_n$ is obtained from $WK_{n-1}$ by adding an adjacent block indicated by two vertical dashed lines. Each such block contains a horizontal braid with two classical and two welded crossings.
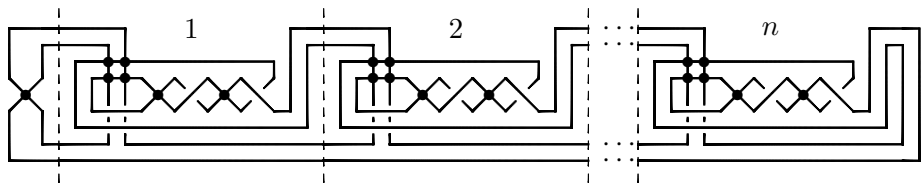
\begin{figure}[ht]
\unitlength=.22mm
\centering 
\begin{picture}(600,100)(-260,-5)
\put(-300,0){\begin{picture}(0,0)
{\thinlines
\qbezier(90,-13)(90,-13)(90,-8)
\qbezier(90,-3)(90,-3)(90,2)
\qbezier(90,7)(90,7)(90,12)
\qbezier(90,17)(90,17)(90,22)
\qbezier(90,27)(90,27)(90,32)
\qbezier(90,37)(90,37)(90,42)
\qbezier(90,47)(90,47)(90,52)
\qbezier(90,57)(90,57)(90,62)
\qbezier(90,67)(90,67)(90,72)
\qbezier(90,77)(90,77)(90,82)
\qbezier(90,87)(90,87)(90,92)
}
{\thinlines
\qbezier(250,-13)(250,-13)(250,-8)
\qbezier(250,-3)(250,-3)(250,2)
\qbezier(250,7)(250,7)(250,12)
\qbezier(250,17)(250,17)(250,22)
\qbezier(250,27)(250,27)(250,32)
\qbezier(250,37)(250,37)(250,42)
\qbezier(250,47)(250,47)(250,52)
\qbezier(250,57)(250,57)(250,62)
\qbezier(250,67)(250,67)(250,72)
\qbezier(250,77)(250,77)(250,82)
\qbezier(250,87)(250,87)(250,92)
\put(170,80){\makebox(0,0)[cc]{$1$}}
}
\thicklines
\qbezier(60,0)(60,0)(90,0)
\qbezier(60,0)(60,0)(60,30)
\qbezier(60,50)(60,50)(60,80)
\qbezier(60,80)(60,80)(90,80)
\qbezier(60,30)(60,30)(80,50)
\qbezier(60,50)(60,50)(80,30)
\qbezier(80,50)(80,50)(80,70)
\qbezier(80,70)(80,70)(90,70)
\qbezier(80,30)(80,30)(80,10)
\qbezier(80,10)(80,10)(90,10)
\put(70,40){\circle*{6}}
\qbezier(90,0)(90,0)(250,0)
\qbezier(90,10)(90,10)(120,10)
\qbezier(120,10)(120,10)(120,16)
\qbezier(120,24)(120,24)(120,26)
\qbezier(120,34)(120,34)(120,70)
\qbezier(120,70)(120,70)(90,70)
\qbezier(80,80)(80,80)(130,80)
\qbezier(130,80)(130,80)(130,34)
\qbezier(130,26)(130,26)(130,24)
\qbezier(130,16)(130,16)(130,10)
\qbezier(130,10)(130,10)(250,10)
\qbezier(100,20)(100,20)(100,60)
\qbezier(100,20)(100,20)(240,20)
\qbezier(240,20)(240,20)(240,70)
\qbezier(240,70)(240,70)(250,70)
\qbezier(100,20)(100,20)(100,60)
\qbezier(100,60)(100,60)(220,60)
\qbezier(110,30)(110,30)(110,50)
\qbezier(110,30)(110,30)(140,30)
\qbezier(110,50)(110,50)(140,50)
\put(120,50){\circle*{6}}
\put(120,60){\circle*{6}}
\put(130,50){\circle*{6}}
\put(130,60){\circle*{6}}
\qbezier(140,30)(140,30)(160,50)
\qbezier(140,50)(140,50)(160,30)
\qbezier(160,30)(160,30)(180,50)
\qbezier(160,50)(160,50)(166,44)
\qbezier(174,36)(180,30)(180,30)
\qbezier(180,30)(180,30)(200,50)
\qbezier(180,50)(180,50)(200,30)
\qbezier(200,50)(200,50)(220,30)
\qbezier(200,30)(200,30)(206,36)
\qbezier(214,44)(220,50)(220,50)
\put(150,40){\circle*{6}}
\put(190,40){\circle*{6}}
\qbezier(220,50)(220,50)(220,60)
\qbezier(220,30)(220,30)(230,30)
\qbezier(230,30)(230,30)(230,80)
\qbezier(230,80)(230,80)(250,80)
\end{picture}}
\put(-140,0){\begin{picture}(0,100)
\thicklines
{\thinlines
\qbezier(250,-13)(250,-13)(250,-8)
\qbezier(250,-3)(250,-3)(250,2)
\qbezier(250,7)(250,7)(250,12)
\qbezier(250,17)(250,17)(250,22)
\qbezier(250,27)(250,27)(250,32)
\qbezier(250,37)(250,37)(250,42)
\qbezier(250,47)(250,47)(250,52)
\qbezier(250,57)(250,57)(250,62)
\qbezier(250,67)(250,67)(250,72)
\qbezier(250,77)(250,77)(250,82)
\qbezier(250,87)(250,87)(250,92)
\put(170,80){\makebox(0,0)[cc]{$2$}}
}
\qbezier(90,0)(90,0)(250,0)
\qbezier(90,10)(90,10)(120,10)
\qbezier(120,10)(120,10)(120,16)
\qbezier(120,24)(120,24)(120,26)
\qbezier(120,34)(120,34)(120,70)
\qbezier(120,70)(120,70)(90,70)
\qbezier(80,80)(80,80)(130,80)
\qbezier(130,80)(130,80)(130,34)
\qbezier(130,26)(130,26)(130,24)
\qbezier(130,16)(130,16)(130,10)
\qbezier(130,10)(130,10)(250,10)
\qbezier(100,20)(100,20)(100,60)
\qbezier(100,20)(100,20)(240,20)
\qbezier(240,20)(240,20)(240,70)
\qbezier(240,70)(240,70)(250,70)
\qbezier(100,20)(100,20)(100,60)
\qbezier(100,60)(100,60)(220,60)
\qbezier(110,30)(110,30)(110,50)
\qbezier(110,30)(110,30)(140,30)
\qbezier(110,50)(110,50)(140,50)
\put(120,50){\circle*{6}}
\put(120,60){\circle*{6}}
\put(130,50){\circle*{6}}
\put(130,60){\circle*{6}}
\qbezier(140,30)(140,30)(160,50)
\qbezier(140,50)(140,50)(160,30)
\qbezier(160,30)(160,30)(180,50)
\qbezier(160,50)(160,50)(166,44)
\qbezier(174,36)(180,30)(180,30)
\qbezier(180,30)(180,30)(200,50)
\qbezier(180,50)(180,50)(200,30)
\qbezier(200,50)(200,50)(220,30)
\qbezier(200,30)(200,30)(206,36)
\qbezier(214,44)(220,50)(220,50)
\put(150,40){\circle*{6}}
\put(190,40){\circle*{6}}
\qbezier(220,50)(220,50)(220,60)
\qbezier(220,30)(220,30)(230,30)
\qbezier(230,30)(230,30)(230,80)
\qbezier(230,80)(230,80)(250,80)
\put(270,80){\makebox(0,0)[cc]{$\cdots$}}
\put(270,70){\makebox(0,0)[cc]{$\cdots$}}
\put(270,10){\makebox(0,0)[cc]{$\cdots$}}
\put(270,0){\makebox(0,0)[cc]{$\cdots$}}
\end{picture}}
\put(50,0){\begin{picture}(0,100)
{\thinlines
\qbezier(90,-13)(90,-13)(90,-8)
\qbezier(90,-3)(90,-3)(90,2)
\qbezier(90,7)(90,7)(90,12)
\qbezier(90,17)(90,17)(90,22)
\qbezier(90,27)(90,27)(90,32)
\qbezier(90,37)(90,37)(90,42)
\qbezier(90,47)(90,47)(90,52)
\qbezier(90,57)(90,57)(90,62)
\qbezier(90,67)(90,67)(90,72)
\qbezier(90,77)(90,77)(90,82)
\qbezier(90,87)(90,87)(90,92)
\put(170,80){\makebox(0,0)[cc]{$n$}}
}
\thicklines
\qbezier(90,0)(90,0)(250,0)
\qbezier(90,10)(90,10)(120,10)
\qbezier(120,10)(120,10)(120,16)
\qbezier(120,24)(120,24)(120,26)
\qbezier(120,34)(120,34)(120,70)
\qbezier(120,70)(120,70)(90,70)
\qbezier(90,80)(90,80)(130,80)
\qbezier(130,80)(130,80)(130,34)
\qbezier(130,26)(130,26)(130,24)
\qbezier(130,16)(130,16)(130,10)
\qbezier(130,10)(130,10)(250,10)
\qbezier(100,20)(100,20)(100,60)
\qbezier(100,20)(100,20)(240,20)
\qbezier(240,20)(240,20)(240,70)
\qbezier(240,70)(240,70)(250,70)
\qbezier(100,20)(100,20)(100,60)
\qbezier(100,60)(100,60)(220,60)
\qbezier(110,30)(110,30)(110,50)
\qbezier(110,30)(110,30)(140,30)
\qbezier(110,50)(110,50)(140,50)
\put(120,50){\circle*{6}}
\put(120,60){\circle*{6}}
\put(130,50){\circle*{6}}
\put(130,60){\circle*{6}}
\qbezier(140,30)(140,30)(160,50)
\qbezier(140,50)(140,50)(160,30)
\qbezier(160,30)(160,30)(180,50)
\qbezier(160,50)(160,50)(166,44)
\qbezier(174,36)(180,30)(180,30)
\qbezier(180,30)(180,30)(200,50)
\qbezier(180,50)(180,50)(200,30)
\qbezier(200,50)(200,50)(220,30)
\qbezier(200,30)(200,30)(206,36)
\qbezier(214,44)(220,50)(220,50)
\put(150,40){\circle*{6}}
\put(190,40){\circle*{6}}
\qbezier(220,50)(220,50)(220,60)
\qbezier(220,30)(220,30)(230,30)
\qbezier(230,30)(230,30)(230,80)
\qbezier(230,80)(230,80)(250,80)
\qbezier(250,10)(250,10)(250,70)
\qbezier(250,0)(250,0)(260,0)
\qbezier(250,80)(250,80)(260,80)
\qbezier(260,0)(260,0)(260,80)
\end{picture}}
\end{picture}
\caption{Welded knot $WK_{n}$.} \label{WKN} 
\end{figure}
Applying a twist move to this braid as in Fig.~\ref{2BRAID}) we will get a braid which can be simplified to a trivial two-strands braid by welded  Reidemeister moves VRII and RII. 
\begin{figure}[ht]
\unitlength=.3mm
\centering 
\begin{picture}(0,40)(0,20)
\put(-160,0){\begin{picture}(0,0)
\thicklines
\qbezier(-40,30)(-40,30)(-20,50)
\qbezier(-20,30)(-20,30)(-40,50)
\qbezier(-20,30)(-20,30)(0,50)
\qbezier(-20,50)(-20,50)(-14,44)
\qbezier(-6,36)(0,30)(0,30)
\qbezier(0,30)(0,30)(20,50)
\qbezier(0,50)(0,50)(20,30)
\qbezier(20,50)(20,50)(40,30)
\qbezier(20,30)(20,30)(26,36)
\qbezier(34,44)(40,50)(40,50)
\put(-30,40){\circle*{4}}
\put(10,40){\circle*{4}}
\end{picture} } 
\put(-90,40){\makebox(0,0)[c,c]{$\Longrightarrow$}}
\put(-90,50){\makebox(0,0)[c,c]{\footnotesize twist}}
\put(-90,30){\makebox(0,0)[c,c]{\footnotesize move}}
\put(0,0){\begin{picture}(0,0)
\thicklines
\qbezier(-60,30)(-60,30)(-40,50)
\qbezier(-40,30)(-40,30)(-60,50)
\put(-50,40){\circle*{4}}
\qbezier(-40,30)(-40,30)(-20,50)
\qbezier(-20,30)(-20,30)(-40,50)
\qbezier(-20,30)(-20,30)(0,50)
\qbezier(-20,50)(-20,50)(-14,44)
\qbezier(-6,36)(0,30)(0,30)
\qbezier(0,30)(0,30)(20,50)
\qbezier(0,50)(0,50)(20,30)
\qbezier(20,30)(20,30)(40,50)
\qbezier(20,50)(20,50)(40,30)
\put(30,40){\circle*{4}}
\qbezier(40,50)(40,50)(60,30)
\qbezier(40,30)(40,30)(46,36)
\qbezier(54,44)(60,50)(60,50)
\put(-30,40){\circle*{4}}
\put(10,40){\circle*{4}}
\end{picture} } 
\put(90,40){\makebox(0,0)[c,c]{$\Longrightarrow$}}
\put(90,50){\makebox(0,0)[c,c]{\footnotesize VRII}}
\put(90,30){\makebox(0,0)[c,c]{\footnotesize RII}}
\put(160,0){\begin{picture}(0,0)
\thicklines
\qbezier(-40,50)(-40,50)(40,50)
\qbezier(-40,30)(-40,30)(40,30) 
\end{picture} } 
\end{picture}
\caption{Applying of a twist move.} \label{2BRAID} 
\end{figure}
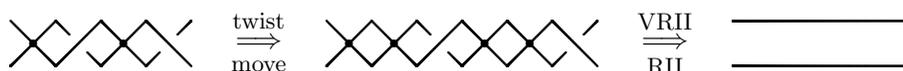
Observe that if the twist move will be applied to a 2-braid of the $m$-th block of the diagram of $WK_{n}$, $m < n$, then in all blocks with numbers $m, m+1, \ldots, n$ 2-braids will became trivial  by repeating of welded Reidemeister moves. 
As a result we obtain $WK_{m-1}$.  Hence we can obtain $WK_m$ from $WK_n$ for $m=0, 1, 2,\ldots, n-1$ by applying a twist move to the 2-braid in $(m+1)$-th block.  This implies $d_T(WK_m,WK_n)\leq 1$. Since $n$ is arbitrary it follows that for any integer distinct $n$ and $m$ we have $d_T(WK_m,WK_n)\leq 1$ as required.
\end{proof}

\begin{problem}
{\rm Are all welded knots $WK_n$ distinct?}  
\end{problem}

\begin{remark}
{\rm We observed that quandle colorings by dihedral quandle $R_n$ fails to distinguish the welded knots $WK_n$ as only possible colorings by $R_n$ happens to be trivial colorings. } 
\end{remark}



\section*{Acknowledgements}

Kaur, Gill and Prabhakar  acknowledge support from the DST project DST/INT/RUS/RSF/P-19. 
Vesnin acknowledges support by the Russian Science Foundation (project 19-41-02005) during the work on Sections 3, 4, 5 and by the Ministry of Science and Higher Education of Russia (agreement No. 075-02-2020-1479/1) for the work on Section 6. 


\end{document}